\documentclass[12pt,reqno]{amsart}
\usepackage{amssymb}

\usepackage{graphicx}
\usepackage{amscd}
\newtheorem{theorem}{Theorem}
\theoremstyle{plain}
\newtheorem{corollary}{Corollary}
\newtheorem{definition}{Definition}
\newtheorem{example}{Example}
\newtheorem{lemma}{Lemme}
\newtheorem{proposition}{Proposition}
\newtheorem{remark}{Remark}
\numberwithin{equation}{section}

\begin{document}
\title[]{On the quasi-isomorphism type of a perfect chain algebra}
\author{Mahmoud Benkhalifa}
\maketitle
\begin{abstract}
Let $R$ be a (P.I.D) and let $T(V),\partial)$ be a free
 $R$-dga. The
quasi-isomorphism type of $(T(V),\partial)$ is the set, denoted
$\{(T(V),\partial)\}$, of all free dgas which are quasi-isomorphic
to $(T(V),\partial)$. In this paper we investigate to characterize
and to compute the set $\{(T(V),\partial)\}$ for a new class of
free dgas called perfect (a special kind of a perfect dga is the
Adams-Hilton model of simply connected CW-complex such that
$H_{*}(X,R)$ is free). We show that if  $(T(V),\partial)$ and
$(T(W),\delta)$ are two perfect dgas, then   $(T(W),\delta)\in
\{(T(V),\partial)\}$
  if and only if their
Whitehead exact sequences are isomorphic. Moreover we show that
every  dga $(T(V),\partial)$ can be split to give a pair
$\big((T(V),\widetilde{\partial}),(\pi_{n})_{n\geq 2}\big)$
consisting with a perfect  dga $(T(V),\widetilde{\partial})$ and a
family of extensions $(\pi_{n})_{n\geq 2}$ and we establish that
if $(T(W),\widetilde{\delta})\in \{(T(V),\widetilde{\partial})\}$
 and if  the extensions $(\pi_{n})_{n\geq 2}$ and $(\pi'_{n})_{n\geq 2}$
 are isomorphic (in a certain sense), then $(T(W),\delta)\in \{(T(V),\partial)\}$.
\end{abstract}

\section{Introduction}
Let $R$ be a principal and integral domain. To each free
differential graded algebra ( free dga for short)
$(T(V),\partial)$ is associated a long exact sequence, called
Whitehead exact sequence:
\begin{equation*}
\cdots \rightarrow H_{n+1}(V,d)\overset{b_{n+1}}{\longrightarrow }%
\Gamma^{T(V)} _{n}\longrightarrow H_{n}(T(V))\longrightarrow H_{n}(V,d)%
\overset{b_{n}}{\longrightarrow }\cdots
\end{equation*}
where $\Gamma^{T(V)} _{n}=\ker \ (H_{n}(T(V_{\leq n}))\rightarrow
V_{n})$ and where $(V,d)$ is the chain complex of the
indecomposables of $(T(V),\partial)$.

Originally this sequence  was  introduced by J.H.C Whitehead in
\cite{W1} (see also \cite{Benk12,Benk18,ben2, BS} ) in order to  classify the homotopy of CW-complexes
simply connected of dimension $4$ and  later   J.H Baues
\cite{Ba3} has shown that this sequence exists also for free dgas (see also \cite{ Benk11,Benk14,Benk15,ben22,ben3,bab} )
and he proved that  two free  dgas $(T(V),\partial)$ and
$(T(W),\delta)$ such that $H_{i}(V,d)=H_{i}(W,d')=0$, for all
$i\geq 4$, are quasi-isomorphic if and if their   Whitehead exact
sequences are isomorphic.

 \noindent Recall that a  dga morphism $\alpha :(T(V),\partial
)\rightarrow (T(W),\delta )$ is called a quasi-isomorphism if
$\alpha$ induces an  isomorphism in homology. In this case we say
that  $(T(V),\partial )$ and $(T(W),\delta )$ are
quasi-isomorphic.

\noindent Recall  also that the quasi-isomorphism type of a free
dga $(T(V),\partial)$ is the set denoted $\{(T(V),\partial)\}$
consisting with all the free dgas which are  quasi-isomorphic to
$(T(V),\partial)$.

Our which would be to extend Baues theorem for a certain class of
free dgas called perfect (see definition 3) which contains all
free dgas $(T(V),\partial)$ such that the graded module
$H_{*}(V,d)$ is free. More precisely we show that:

 \noindent \emph{ Two perfect dgas are   quasi-isomorphic if and only if
their Whitehead exact sequences are isomorphic.}

\noindent As a consequence we establish the following topological
result:

\noindent \emph{ If  $X$ and  $Y$ are two simply connected
CW-complex such that  $H_{*}(X,R)$ and $H_{*}(Y,R)$
 are free,
then their respective  Adams-Hilton models \cite {ada} are
 quasi-isomorphic if and only if the  Whitehead exact sequences
associated to these models  are isomorphic.}

 \noindent Moreover we describe an algorithm that allows us to
compute the  set  $\{(T(V),\partial)\}$ where $(T(V),\partial)$ is
a perfect dga.

Next in order to generalize the above results we show that every
free dga $(T(V),\partial )$ splits to give a pair
$\big((T(V),\overline{\partial}),(\pi_{n})_{n\geq 2}\big)$
consisting with  a perfect dga $(T(V),\overline{\partial})$ and a
family of extensions $(\pi_{n})_{n\geq 2}$, where $\pi_{n}\in
\mathrm{Ext}(H_{n}(V,d),\mathrm{Coker}\,b_{n+1})$ and we establish
that:

 \emph{If  the perfect dgas $(T(V),\overline{\partial})$ and $(T(W),\overline{\delta})$ are quasi-isomorphic  and
if  the extensions $(\pi_{n})_{n\geq 2}$ and $(\pi'_{n})_{n\geq
2}$
 are isomorphic (in a certain sense), then $(T(V),\partial)$ and $(T(W),\delta)$ are
 quasi-isomorphic.}

This article is organized as follows. In section 2, the Whitehead
exact sequence associated with a free dgas  as well as  perfect
dgas are defined and their essential properties are given. Section
3 is devoted to the notion of adapted systems which constitutes
the technical part on which this work is based and that allows us
to derive  the main first results in this paper. At the end  of
section 3 we give an algorithm showing how the results in this
section can be used in order to compute the set of homotopy types
of  perfect dgas. In section 4 is devoted to the proof of the
second main result and in the end of this paper we illustrate our
results by giving some  geometric applications.
\section{Whitehead exact sequence}
Let $(T(V),\partial )$ be a free dga. For all $n\geq 2$, let
$T(V_{\leq n})$ be the free sub-dga of $T(V)$ generated by the
graded module $(V_{i})_{i\leq n}$. Define the pair $(T(V_{\leq
n}),T(V_{\leq n-1}))$ as the quotient:
\begin{equation*}
(T(V_{\leq n}),T(V_{\leq n-1}))=\frac{T(V_{\leq n})}{T(V_{\leq
n-1})}.
\end{equation*}
From the following long exact sequence associated to the pair
$(T(V_{\leq n}),T(V_{\leq n-1}))$:

\begin{picture}(300,40)(10,67)
\put(25,80){$ \cdots \rightarrow H_{n+1}(T(V_{\leq n}),T(V_{\leq
n-1}))\overset{\phi_{n+1}}{\longrightarrow} H_{n}(T(V_{\leq
n-1}))\overset{i_{n}}{\longrightarrow} H_{n}(T(V_{\leq n}))$}
\put(330,76){$\vector(0,-1){30}$} \put(70,35){$\cdots \leftarrow
H_{n-1}(T(V_{\leq n-1}))\overset{\beta_{n}}{\longleftarrow }
H_{n}(T(V_{\leq n}),T(V_{\leq n-1}))\cong V_{n}$}
\put(332,60){\scriptsize $j_{n}$}
\end{picture}

\noindent where the connecting $\beta _{n}$ is defined by:
\begin{equation}
\beta_{n}(v_{n})=\partial_{n}(v_{n})+\mathrm{Im}\,\partial_{n-1},\label{d1}
\end{equation}
 we define the graded module
$(\Gamma _{n}^{T(V)})_{n\geq 2}$ by setting:
\begin{equation}
\Gamma^{T(V)} _{n}=\ker j_{n}=\mathrm{Im}\,i_{n}.\label{X12}
\end{equation}
\begin{remark}
\label{29} If $V_{1}=0$, then $\Gamma^{T(V)} _{n}=H_{n}(T(V_{\leq
n-1})$ because in this case we have $H_{n+1}(T(V_{\leq
n}),T(V_{\leq n-1}))=0.$

From the relation (\ref{X12}) and the above exact sequence we
deduce that:
\begin{equation}
\Gamma^{T(V)} _{n}=\frac{H_{n}(T(V_{\leq
n-1}))}{\mathrm{Im}\,\phi_{n+1}}\label{101}
\end{equation}
so if  $Z_{n}(T(V_{\leq n-1}))$ denotes the sub-module of the
$n$-cycles of $(T_{n}(V_{\leq n-1}))$, then we have the following
projections:
\begin{equation}
Z_{n}(T(V_{\leq n-1}))\overset{}{\rightarrow}H_{n}(T(V_{\leq
n-1}))\overset{}{\rightarrow}\Gamma^{T(V)} _{n}\label{32}
\end{equation}
\end{remark}
The linear part $d$ of differential
  $\partial$, which is a  differential on the graded  module
 $V$,  satisfies the relation:
\begin{equation}
d_{n}=j_{n-1}\circ \beta _{n},\,\,\,\,\forall n\geq 2, \label{26}
\end{equation}
 so  the chain  complex
  $(V,d)$ can be identified with the graded module of the
indecomposables \cite{Ba3,Benk10,Benk13,Benk16} of $(T(V),\partial)$. Hence it results
the graded module: 
$$
H_{\ast}(V,d)=\Big(\frac{\ker
d_{n+1}}{\mathrm{Im}\,d_{n+2}}\Big)_{n\geq 1}
$$

Recall that  if $\alpha:(T(V),\partial )\rightarrow (T(W),\delta
)$ is a dga-morphism,  then  $\alpha$ induces a chain map
$\widetilde{\alpha}:(V,d)\to (W,d')$ which implies the graded
homomorphism:
\begin{equation}
H_{\ast}(\widetilde{\alpha}):H_{\ast}(V,d)\to
H_{\ast}(W,d').\label{28}
\end{equation}

The Whitehead exact sequence associated with $(T(V),\partial )$ is
by definition  the following exact long sequence:
\begin{equation*}
\cdots \rightarrow H_{n+1}(V,d)\overset{b_{n+1}}{\longrightarrow }%
\Gamma^{T(V)} _{n}\longrightarrow H_{n}(T(V))\longrightarrow H_{n}(V,d)%
\overset{b_{n}}{\longrightarrow }\cdots
\end{equation*}
where :
\begin{equation}
b_{n+1}\circ pr(z)=\beta _{n+1}(z). \label{8}
\end{equation}
Here  $pr:\ker d_{n+1}\twoheadrightarrow H_{n+1}(V,d)$ denotes the
projection.
\begin{remark}
\label{50} Since $V_{n}$ is free, the homomorphism  $j_{n}:
H_{n}(T(V_{\leq n}))\to V_{n}$ implies that $H_{n}(T(V_{\leq
n}))\cong \ker j _{n}\oplus \mathrm{Im}\, j _{n}$ and since
$\mathrm{Im}\, j _{n}=\ker \beta _{n}$ and $\Gamma^{T(V)}
_{n}=\ker j _{n}$ we deduce the following isomorphism:
\begin{equation}
 H_{n}(T(V_{\leq n}))\overset{i}{\cong}\Gamma^{T(V)} _{n}\oplus
\ker \beta _{n} \label{4}
\end{equation}

\noindent The isomorphism $i$ is defined as follows: first we
observe  that any  n-cycle   in  $T_{n}(V_{\leq n})$ can be
expressed as $v_{n}+q_{v_{n}}+\widetilde{q}_{v_{n}}$, where
$\partial_{n}(v_{n}+q_{v_{n}})=0$  and where
$\partial_{n}(\widetilde{q}_{v_{n}})= 0$. Here  $v_{n}\in V_{n}$
and $q_{v_{n}},\widetilde{q}_{v_{n}}\in T_{n}(V_{\leq n-1}))$.

\noindent Since $\partial_{n}(\widetilde{q}_{v_{n}})= 0$, the
element $\widetilde{q}_{v_{n}}+ \mathrm{Im}\,\partial_{n+1}$ is in
$H_{n}(T(V_{\leq n-1}))$ and therefore the element
$(\widetilde{q}_{v_{n}}+ \mathrm{Im}\,\partial_{n+1}+\mathrm{Im}\,
\phi _{n+1})$ is in $\Gamma^{T(V)} _{n}$ according to formula
(\ref{101}). Hence  the isomorphism $i$ is defined by setting:
\begin{equation}
i(v_{n}+q_{v_{n}}+\widetilde{q}_{v_{n}}+\mathrm{Im}\,
\partial _{n+1})= v_{n}\oplus (\widetilde{q}_{v_{n}}+
\mathrm{Im}\,\partial_{n+1}+\mathrm{Im}\, \phi _{n+1}) \label{102}
\end{equation}

\noindent Moreover it is well  know that a dga-morphism
$\alpha:(T(V),\partial)\to (T(W),\delta)$  induces the following
commutative diagram:

 \begin{picture}(300,90)(0,35)
\put(90,102){$\Gamma^{T(V)}_{n}\rightarrowtail H_{n}(T(V_{\leq
n}))\twoheadrightarrow \ker \beta_{n}$}
\put(90,48){$\Gamma^{T(W)}_{n}\rightarrowtail H_{n}(T(W_{\leq
n}))\twoheadrightarrow \ker \beta'_{n}$}
 \put(94,98){$\vector(0,-1){39}$}
 \put(160,98){$\vector(0,-1){39}$}
 \put(162,78){$H_{n}(\alpha^{n})$}
\put(232,78){$\widetilde{\alpha}_{n}$}
\put(77,78){$\gamma^{\alpha^{n}}_{n}$}
\put(230,98){$\vector(0,-1){39}$}
\end{picture}

\noindent where $\alpha^{n}:(T(V_{\leq n}),\delta^{(n)})\to
(T(W_{\leq n}),\delta^{(n)})$, where $\widetilde{\alpha}^{n}$ is
given by the relation (\ref{28}) and where
$\gamma^{\alpha^{n}}_{n}$ is the restriction of
$H_{n}(\alpha^{n})$ to $\Gamma_{n}^{T(V)}.$ Therefore according to
the relation (\ref{4}) the homomorphism:
$$H_{n}(\alpha^{n}):H_{n}(T(V_{\leq n}))\cong \Gamma^{T(V)} _{n}\oplus
\ker \beta _{n}\longrightarrow H_{n}(T(W_{\leq
n}))\cong\Gamma^{T(W)} _{n}\oplus \ker \beta' _{n}$$
 splits into:
\begin{equation}
H_{n}(\alpha^{n})=\gamma^{\alpha^{n}}_{n}\oplus
\widetilde{\alpha}_{n}\label{a31}
\end{equation}

- In terms of the differential $d_{n+1}:V_{n+1}\rightarrow V_{n}$
we deduce the following decomposition:
\begin{equation}
V_{n+1}\cong (\mathrm{Im}\,d_{n+1})^{\prime }\oplus \ker d_{n+1}
\label{a1}
\end{equation}
where $(\mathrm{Im}\,d_{n+1})^{\prime }\subset V_{n+1}$ is a copy of $\mathrm{%
Im}\,d_{n+1}\subset V_{n}$. Therefore the short exact sequence:
\begin{equation}
(\mathrm{Im}\,d_{n+1})^{\prime }\overset{d_{n+1}}{\rightarrowtail
}\ker d_{n}\twoheadrightarrow H_{n}(V,d)\label{19}
\end{equation}
may be chosen as a free resolution of the module $H_{n}(V,d)$ and
since $d_{n+1}( (\mathrm{Im}\,d_{n+1})^{\prime }) \subset \ker
\beta _{n}$, the short exact sequence:
\begin{equation}
(\mathrm{Im}\,d_{n+1})^{\prime }\overset{d_{n+1}}{\rightarrowtail
}\ker \beta _{n}\twoheadrightarrow \ker b_{n}  \label{100}
\end{equation}
may be also chosen as a free resolution of the sub-module $\ker
b_{n}\subset H_{n}(V,d).$
\end{remark}
\begin{remark}
\label{20}  Let $(T(V),\partial)$ be a free dga and let
$(l_{n+1,\sigma ^{\prime }})_{\sigma ^{\prime }\in \sum^{\prime
}}$  a chosen  basis of the free sub-module
$(\mathrm{Im}\,d_{n+1})'$. By the formula (\ref{26}) we know that
$\beta _{n+1}(l_{n+1,\sigma ^{\prime
}})=\partial_{n+1}(l_{n+1,\sigma'})+\mathrm{Im}\,\partial_{n}$ and
by remark \ref{50} we know that the $n$-cycle
$\partial_{n+1}(l_{n+1,\sigma ^{\prime }})$ can be written as
$\partial_{n+1}(l_{n+1,\sigma ^{\prime }})=d_{n+1}(l_{n+1,\sigma
'})+q_{l_{n+1,\sigma '}}+\widetilde{q}_{l_{n+1,\sigma '}}$.
Therefore according to isomorphism $i$ defined in (\ref{102}) we
get:
\begin{eqnarray}
i\circ \beta_{n+1}(l_{n+1,\sigma ^{\prime
}})=d_{n+1}(l_{n+1,\sigma '})\oplus (\widetilde{q}_{l_{n+1,\sigma
'}}+ \mathrm{Im}\,\partial_{n+1}+\mathrm{Im}\, \phi
_{n+1})\label{13}
\end{eqnarray}

\noindent Define the homomorphism  $\varphi
_{n}:(\mathrm{Im}\,d'_{n+1})\rightarrow \Gamma^{T(V)}_{n}$ by
setting $\varphi _{n}(l_{n+1,\sigma ^{\prime
}})=(\widetilde{q}_{l_{n+1,\sigma '}}+
\mathrm{Im}\,\partial_{n+1}+\mathrm{Im}\, \phi _{n+1})$. Therefore
by using the resolutions
of $H_{n}(V,d)$ and   $\ker b_{n}$,  given  respectively  in (\ref{19}) and (\ref{100}), the homomorphism $\overline{\varphi _{n}}:(\mathrm{Im}\,d_{n+1})'\overset{\varphi _{n}%
}{\rightarrow }\Gamma^{T(V)} _{n}\overset{pr}{\twoheadrightarrow} \mathrm{Coker}%
\,b_{n+1}$, where the homomorphism $b_{n+1}$ is given by the
Whitehead exact sequence associated with $(T(V),\partial)$,
provides two extensions which are :
\begin{eqnarray}
 [\overline{\varphi _{n}}]\in Ext_{R}^{1}(H_{n},\mathrm{Coker}%
\,b_{n+1})=\frac{Hom\big((\mathrm{Im}\,d_{n+1})^{\prime },\mathrm{Coker}%
\,b_{n+1}\big)}{(d_{n+1})^{*}\big(Hom(\ker d_{n},\mathrm{Coker}%
\,b_{n+1})\big)}\label{14}\\
 \{\overline{\varphi _{n}}\}\in Ext_{R}^{1}(\ker b_{n},\mathrm{Coker}%
\,b_{n+1})=\frac{Hom\big((\mathrm{Im}\,d_{n+1})^{\prime },\mathrm{Coker}%
\,b_{n+1}\big)}{(d_{n+1})^{*}\big(Hom(\ker d_{n},\mathrm{Coker}%
\,b_{n+1})\big)}\label{23}
 \end{eqnarray}
 \end{remark}
\noindent Thus  remark \ref{20} allows us to suggest  the
following definitions:
\begin{definition} \label{d3} A free dga $(T(V),\partial)$  is called $n$-perfect if the
homomorphism  $\varphi _{n}$ is trivial.
\end{definition}
\begin{definition}
\label{dd3} A free dga $(T(V),\partial)$ is called quasi
$n$-perfect if the extension $[\overline{\varphi _{n}}]$ is
trivial.
\end{definition}
\begin{remark}
\label{40}Consider the resolution defined  in (\ref{19}), to say
that the extension $[\overline{\varphi _{n}}]$ is  trivial means
that there exists a homomorphism $g_{n} :\ker d _{n}\rightarrow
\mathrm{Coker}\,b_{n+1}$ satisfying  the relation:
\begin{equation}
 \overline{\varphi
_{n}}=g_{n}\circ d_{n+1}\label{41}
\end{equation}
and  since $\ker \beta _{n}\subseteq \ker d _{n}$, we
deduce that the extension $\{\overline{\varphi _{n}}\}$ is also  trivial in $Ext_{R}^{1}(\ker b_{n},\mathrm{Coker}%
\,b_{n+1})$.\\

\noindent If $(T(V),\partial)$ is $n$-perfect and  if
$(z_{n+1,\sigma })_{\sigma \in \sum}$  is a chosen  basis of the
free sub-module $\ker d_{n+1})'$, then  the formulas (\ref{8}) and
(\ref{13}) imply that:
\begin{eqnarray}
i\circ \beta_{n+1}(z_{n+1,\sigma}+l_{n+1,\sigma '})=b_{n+1}\circ
pr(z_{n+1,\sigma})\oplus d_{n+1}(l_{n+1,\sigma '})\label{a13}
\end{eqnarray}

\end{remark}
\begin{example}
\label{36}Every free dga $(T(V),\partial)$ such that:
$$H_{3}(V,d)=\Bbb Z_{p}\hspace{1cm}H_{2}(V,d)=\Bbb
Z_{q}\hspace{1cm}H_{1}(V,d)=\Bbb Z_{s}$$
 where $p$, $q$ et $s$ are relatively  prime, is $2$-perfect . This comes immediately from the  fact that $Hom(H_{3}(V,d), \Gamma_{2}^{T(V)})=0$, so $Ext_{R}^{1}(H_{2}(V,d),\mathrm{Coker}%
\,b_{3})=Ext_{R}^{1}(H_{2}(V,d), \Gamma_{2}^{T(V)})\\=0$. Recall
that  form its definition it is easy to show that:
\begin{eqnarray*}
\Gamma_{2}^{T(V)}=H_{1}(V,d)\otimes H_{1}(V,d)
\end{eqnarray*}
\end{example}
\begin{definition}
\label{16} A free dga $(T(V),\partial)$ is called perfect  if it
is  $n$-perfect  for all $n$.
\end{definition}
\begin{definition}
\label{d16} A free dga $(T(V),\partial)$ is called quasi-perfect
if it is  quasi $n$-perfect  for all $n$.
\end{definition}
 From the above definitions it is clear that every
perfect  dga  $(T(V),\partial)$ is quasi-perfect .
\begin{example}
\label{35}If  $(T(V),\partial)$  is free dga  such that
 $H_{*}(V,d)$ is free,  then $(T(V),\partial)$  is
perfect  because in this case we have $Ext_{R}^{1}(H_{*}(X,\Bbb Z),\mathrm{Coker}%
\,b_{*+1})=0$. Thus  by using the  notion of the minimal model for
dgas defined in \cite{BL} we conclude that every  dga over a field
is  quasi-isomorphic  to a quasi-perfect dga.
\end{example}
\begin{example}
\label{39} Define the free graded  abelian  $V$ par the following
relations:
\begin{eqnarray*}
V_{4}=<e_{1},e_{2}>\,\,\,,\,\,\, V_{3}=<c_{1},c_{2}> \,\,\,,\,\,\,
V_{2}=<b_{1},b_{2}>\,\,\,,\,\,\, V_{1}=<a_{1},a_{2}>\,\,\,,\,\,\,
V_{i}=0\,\,\,,\,\,\, i\geq 5
\end{eqnarray*}
\noindent Next  define  the   differentials $\partial$, $\delta$
and $\psi$ on $T(V)$ as follows:
$$
\begin{array}{lcrlcrlcrlcr}
\partial_{4}(e_{1})=2c_{2}&\hspace{5mm}\partial_{3}(c_{2})&=&0&\hspace{5mm}\partial_{2}(b_{1})&=&0&\hspace{5mm}\partial_{1}(a_{1})&=&0,\\
\partial_{4}(e_{2})=b_{1}\otimes a_{1}&\hspace{5mm}\partial_{3}(c_{1})&=&2a_{1}\otimes a_{1}&\hspace{5mm}\partial_{2}(b_{2})&=&a_{2}&\hspace{5mm}\partial_{1}(a_{2})&=&0\\\\
\delta_{4}(e_{1})=2c_{2}&\hspace{1cm}\delta_{3}(c_{2})&=&0&\hspace{1cm}\delta_{4}(b_{1})&=&0&\hspace{5mm}\partial_{1}(a_{1})&=&0,\\
\delta_{4}(e_{2})=b_{1}\otimes
a_{1}&\hspace{1cm}\delta_{3}(c_{1})&=&a_{1}\otimes
a_{1}&\hspace{1cm}\partial_{2}(b_{2})&=&a_{2}&\hspace{5mm}\partial_{1}(a_{2})&=&0\\\\
\psi_{4}(e_{1})=2c_{2}&\hspace{1cm}\psi_{3}(c_{2})&=&0&\hspace{1cm}\psi_{2}(b_{1})&=&0&\hspace{5mm}\psi_{1}(a_{1})&=&0,\\
\psi_{4}(e_{2})=b_{1}\otimes
a_{1}&\hspace{1cm}\psi_{3}(c_{1})&=&a_{1}\otimes
a_{1}&\hspace{5mm}\psi_{2}(b_{2})&=&3a_{2}&\hspace{5mm}\psi_{1}(a_{2})&=&0
\end{array}
$$

\noindent For the dga $(T(V),\partial)$ it is clear that:
$$H_{3}(V,d)=\Bbb {Z}_{2}\hspace{1cm}
 H_{2}(V,d)=\Bbb Z_{2}\hspace{1cm}
 H_{1}(V,d)=\Bbb Z\hspace{1cm}
 \ker d_{2}=\Bbb Z$$
Consider the homomorphism
$\varphi_{2}:(\mathrm{Im}\,d_{3})'\rightarrow
\mathrm{Coker}\,b_{3}$. First recall that we have:
$$\Gamma_{2}^{(T(V),\partial)}=H_{1}(V,d)\otimes H_{1}(V,d)=\Bbb
Z\hspace{2cm}(\mathrm{Im}\,d_{3})'\cong \Bbb Z.$$ The homomorphism
$b_{3}:\Bbb Z_{2}=H_{3}(V,d)\rightarrow
\Gamma_{2}^{(T(V),\partial)}=\Bbb Z$ is  obviously nil therefore
$\mathrm{Coker}\,b_{3}=\Gamma_{2}^{(T(V),\partial)}=\Bbb Z$.

\noindent  Now since
 the homomorphism $\varphi_{2}$ is defined by $\varphi_{2}(c_{1})=\overline{2a_{1}\otimes
a_{1}}$, where $\overline{2a_{1}\otimes a_{1}}\in H_{2}(T(V_{\leq
2}))$ denotes the homology class of the 2-cycle $2a_{1}\otimes
a_{1}$, we deduce that  the homomorphism $\varphi_{2}:\Bbb Z\to
\Bbb Z$ is the multiplication by 2. So if we take the identity
homomorphism $Id_{\Bbb Z}$, then we can make the following diagram
commutative:

\begin{picture}(300,90)(-80,35)
\put(60,98){$\vector(0,-1){39}$}\put(12,100){$(\mathrm{Im}\,d_{3})'\cong
\Bbb Z\vector(1,0){70}\hspace{2mm}\Bbb Z\cong \ker d_{2}$}
\put(10,50){$\mathrm{Coker}\,b_{4}\cong\Bbb Z$}
\put(160,96){$\vector(-2,-1){79}$} \put(19,78){\scriptsize
$\varphi_{3}=2\times.$} \put(93,103){\scriptsize $d_{3}=2\times.$}
\put(130,70){\scriptsize $Id_{\Bbb Z}$}
\end{picture}

\noindent It results that the extension $[\varphi_{2}]$ is
trivial,  therefore $(T(V),\partial)$ is quasi 2-perfect. Note
that $(T(V),\partial)$ is not 2-perfect because the homomorphism
$\varphi_{2}$ is not trivial.

 For the dga $(T(V),\delta)$ a similar computation shows that:
$$H_{3}(V,d')=\Bbb {Z}_{2}\hspace{1cm}
 H_{2}(V,d')=\Bbb Z_{2}\hspace{1cm}
 H_{1}(V,d')=\Bbb Z\hspace{1cm}
 \ker d'_{2}=\Bbb Z$$
 then  the homomorphism $b_{3}$ is nil and  the homomorphism
$\varphi_{2}:(\mathrm{Im}\,d_{3})'\cong \Bbb Z\rightarrow \Bbb
Z\cong \mathrm{Coker}\,b_{3}$ is identified to the identity
homomorphism $id_{\Bbb Z}$, so the extension $[\varphi_{2}]$ in
$Ext(H_{2}(V,d),\mathrm{Coker}\,b_{3})=Ext(\Bbb Z_{2},\Bbb Z)=\Bbb
Z_{2}$ is not trivial therefore  $(T(V),\delta)$ is not quasi
2-perfect.

 For the dga $(T(V),\psi)$ we have in this case:
$$H_{3}(V,d'' )=\Bbb {Z}_{2}\hspace{5mm}
 H_{2}(V,d'')=\Bbb Z_{2}\hspace{5mm}
 H_{1}(V,d'')=\Bbb Z_{2}\hspace{5mm}
 \ker d''_{2}=\Bbb Z\hspace{5mm}
\Gamma_{2}^{(T(V),\psi)}=\Bbb Z_{2}$$ The homomorphism $b_{3}$ is
not nil  therefore $\mathrm{Coker}\,b_{3}=0$, so the homomorphism
$\varphi_{2}$ is also nil. Hence  $(T(V),\psi)$ is 2-perfect.
\end{example}
\begin{proposition}
 If $(T(V),\partial)$ is a  quasi-perfect dga, then  we have:
\begin{equation}
H_{n}(T(V))\cong \mathrm{Coker}\,b_{n+1}\oplus \ker b
_{n},\hspace{5mm}\forall n\geq 2\label{L6}
\end{equation}
\end{proposition}

\begin{proof}
On one hand  we remark  that the following short exact sequence:
$$\mathrm{Coker}%
\,b_{n+1}\rightarrowtail H_{n}(T(V))\twoheadrightarrow \ker
b_{n}$$
  extracted from the Whitehead exact sequence associated with $(T(V),\partial)$ defines an extension
$ [H_{n}(T(V))]\in Ext(\ker b_{n},\mathrm{Coker}%
\,b_{n+1})$ and on the other hand from the  exact sequence
$V_{n+1}\overset{\beta _{n+1}}{\longrightarrow }H_{n}(T(V_{\leq
n}))\rightarrow H_{n}(T(V_{\leq n+1}))\rightarrow 0$ and the
relation (\ref{4}) we get:
\begin{equation}
H_{n}(T(V_{\leq n+1}))\cong \frac{\Gamma^{T(V)} _{n}\oplus \ker
\beta _{n}}{ \mathrm{Im}\,\beta _{n+1}}.  \label{112}
\end{equation}
\noindent Substituting the relation (\ref{13}), the  formula
(\ref{112}) becomes:
\begin{equation}
H_{n}(T(V_{\leq n+1}))\cong \frac{\Gamma^{T(V)} _{n}\oplus \ker
\beta _{n}}{\mathrm{Im}\varphi _{n}+\mathrm{Im}\,b_{n+1}\oplus
\mathrm{Im}\,d_{n+1}}\label{115}
\end{equation}
\noindent But it is well-known that $H_{n}(T(V_{\leq
n+1}))=H_{n}(T(V))$ and:
\begin{equation*}
\frac{\Gamma^{T(V)} _{n}}{\mathrm{Im}\varphi
_{n}+\mathrm{Im}\,b_{n+1}}\cong
\frac{\mathrm{Coker}\,b_{n+1}}{\mathrm{Im}\overline{\varphi _{n}}}
\end{equation*}
via the isomorphism  sending the element $x+(\mathrm{Im}\varphi
_{n}+\mathrm{Im}\,b_{n+1})$, where $x\in \Gamma^{T(V)} _{n}$, on
$(x+\mathrm{Im}\,b_{n+1})+\mathrm{Im}\overline{\varphi} _{n}$, so
the relation (\ref{115}) may be written:
\begin{equation*}
H_{n}(T(V))\cong \frac{\mathrm{Coker}\,b_{n+1}\oplus \ker \beta
_{n}}{\mathrm{Im}\overline{\varphi
_{n}}\oplus\mathrm{Im}\,d_{n+1}}\label{43}.
\end{equation*}
The last expression means that the following square:

\begin{picture}(300,100)(10,35)
\put(120,98){$\vector(0,-1){39}$} \put(95,102){$
(\mathrm{Im}\,d_{n+1})'\overset{d_{n+1}}{\rightarrowtail}\ker
\beta_{n}\twoheadrightarrow \ker b_{n}$}
\put(92,50){$\mathrm{Coker}
\,b_{n+1}\overset{}{\rightarrowtail}H_{n}(T(V))\twoheadrightarrow
\ker b_{n}$} \put(121,78){$\overline{\varphi _n}$}
\put(180,98){$\vector(0,-1){39}$}
\end{picture}

\noindent is a push out.  Therefore by the isomorphism $$\Phi:Ext_{R}^{1}(\ker b_{n},\mathrm{Coker}%
\,b_{n+1})\overset{\cong}{\longrightarrow}Ext(\ker b_{n},\mathrm{Coker}%
\,b_{n+1}),$$ we deduce that $[H_{n}(T(V))]=
\Phi(\{\overline{\varphi _{n}}\})$ for all $n\geq 2$, where the
extension $\{\overline{\varphi _{n}}\}$ is defined in remark 2.
Now since $(T(V),\partial)$ is quasi-perfect , the extension
$[\overline{\varphi _{n}}]$ is  trivial  in $Ext_{R}^{1}(H_{n}(V,d),\mathrm{Coker}%
\,b_{n+1})$, which implies that the extension
$\{\overline{\varphi _{n}}\}$  is also trivial in $Ext_{R}^{1}(\ker b_{n},\mathrm{Coker}%
\,b_{n+1})$. It results from this that the extension
$[H_{n}(T(V))]$ is trivial.
\end{proof}
\section{Adapted  systems  of order $n$}
This section is  the technical part  of this work. We begin by
defining the adapted systems of order $n$ and their morphisms, a
notion  which we need  to prove the main results in this paper.
\begin{definition}
 An adapted  system  of order $n$ is a pair
$\big((T(V),\partial^{(n)}),b_{n+1}\big)$ consisting with a
perfect dga  $(T(V_{\leq n}),\partial^{(n)})$, a free  chain
complex $(V,d)$  such that the composed homomorphism:
\begin{equation}
V_{n+1}\overset{d_{n+1}}{\longrightarrow } H _{n}(T(V_{\leq
n}),T(V_{\leq n-1}))\cong V_{n}\overset{\beta
_{n}}{\longrightarrow }H _{n-1}(T(V_{\leq n-1})) \label{22}
\end{equation}
is trivial and a homomorphism $b_{n+1}:H_{n+1}(V,d) \rightarrow
\Gamma^{T(V_{\leq n})}_{n}$. Recall that $\Gamma^{T(V_{\leq
n})}_{n}$ is given by the relation (\ref{4}).
\end{definition}
\begin{definition}
A morphism  between two adapted   systems  of order $n$ is a pair
 $(\xi _{*},\alpha ):\big((T(V),\partial^{(n)}),b_{n+1}\big)\to
\big((T(W),\delta^{(n)}),b'_{n+1}\big)$  with the following
properties:

 \noindent $\alpha:(T(V_{\leq n}),\partial^{(n)})\to(T(W_{\leq n}),\delta^{(n)})$ is  a dga-morphism  and  $\xi _{*}:(V,d)\to(W,d')$ is a chain map satisfying $H_{\leq
n}(\xi _{*})=H_{\leq n}(\alpha) $  and for which the following
diagram commutes:

\begin{picture}(300,90)(-30,30)
\put(65,100){$ H_{n+1}(V,d)
\hspace{1mm}\vector(1,0){80}\hspace{1mm}H_{n+1}(W,d')$}
 \put(89,76){$b_{n+1}$} \put(217,76){$b'_{n+1}$}
\put(87,97){$\vector(0,-1){38}$} \put(215,96){$\vector(0,-1){38}$}
\put(145,103){\scriptsize $H_{n+1}(\xi_{*})$}
\put(150,52){$\gamma_{n}^{\alpha}$} \put(85,48){$
\Gamma^{T(V_{\leq n})}_{n}\hspace{1mm}\vector(1,0){84}\hspace{1mm}
\Gamma ^{T(W_{\leq n})}_{n}$} \put(-15,76){$(A)$}
\end{picture}

\noindent where $\gamma_{n}^{\alpha}$ is the restriction of
$H_{n}(\alpha)$ to $\Gamma^{T(V_{\leq n})}_{n}$. Note  that  the
commutativity of the diagram $(A)$ means that:
\begin{equation}
\gamma_{n}^{\alpha}\circ b_{n+1}\circ pr( z)=b'_{n+1}\circ pr\circ
\xi_{n+1}(z)\label{103}
\end{equation}
 where  $z\in \ker d_{n+1}$ and where $pr:\ker
d_{n+1}\twoheadrightarrow H_{n+1}(V,d)$, $pr':\ker
d'_{n+1}\twoheadrightarrow H_{n+1}(W,d')$
\end{definition}
 \begin{example}
\label{e1}Let $(T(V),\partial)$ be a  perfect dga and let:
\begin{equation*}
\cdots \rightarrow H_{n+1}(V,d)\overset{b_{n+1}}{\longrightarrow }%
\Gamma^{T(V)} _{n}\longrightarrow H_{n}(T(V))\longrightarrow H_{n}(V,d)%
\overset{b_{n}}{\longrightarrow }\cdots
\end{equation*}
be its  Whitehead exact sequence. Then  the pair
$\big((T(V),\partial^{(n)}),b_{n+1}\big)$, where $(V,d)$ is the
module of the indecomposables of  $(T(V),\partial)$, is an adapted
system of order $n$. Likewise if
$\alpha:(T(V),\partial)\rightarrow (T(W),\delta)$ is a dga
morphism, then the pair $(\widetilde{\alpha} ,\alpha ^{ n})$,
where $\alpha^{ n}:(T(V_{\leq n}),\partial^{(n)})\rightarrow
(T(W_{\leq n}),\delta^{(n)})$ is the restriction of $\alpha$ to
$T(V_{\leq n})$ and where $\widetilde{\alpha}:(V,d)\to (W,d')$ is
like in (\ref{28}), is a morphism from the two adapted systems of
order $n$ associated respectively with $(T(V),\partial)$ and
$(T(W),\delta)$.
\end{example}
\begin{definition}
\label{r3}  We say that a  morphism  $(\xi _{*},\alpha )$ is an
equivalence  if $\alpha $ is a quasi-isomorphism and  if
$H_{*}(\xi _{*})$ is an  isomorphism of  graded modules.
\end{definition}

Let us denote by $\mathbf {PDGA}$  the category of perfect dgas,
by $\mathbf {PDGA}_{n+1}$ the full sub-category of $\mathbf
{PDGA}$ consisting with   $(T(V),\partial)$ such that $V_{i}=0$
for all $i\geq n+2$, by $\Bbb {AD}_{n}$ the category of the
adapted systems of order $n$ and their morphisms and by $\Bbb
{AD}_{n}^{n+1}$ the full sub-category of $\Bbb {AD}_{n}$ whose
objects are the adapted systems
$\big((T(V),\partial^{(n)}),b_{n+1}\big)$ such that $V_{i}=0$ for
all $i\geq n+2$. Example \ref{e1} allows us to define, for all
$n\geq 2$, a functor $F_{n}:\mathbf {PDGA}\to\Bbb {AD}_{n}$ by
setting:
\begin{eqnarray}
F_{n}(X)=\big((T(V),\partial^{(n)}),b_{n+1}\big)\hspace{5mm},\hspace{5mm}F_{n}(\alpha)=(\widetilde{\alpha},\alpha
^{ n})\label{15}
\end{eqnarray}
and let $F_{n}^{n+1}:\mathbf {PDGA}_{n+1 }\to\Bbb {AD}_{n}^{n+1}$
be its  restriction  to $\mathbf {PDGA}_{n+1}$.
\subsection{Properties of the functor $F_{n}$}
This  paragraph is devoted to  the properties of the functor
$F_{n}$ which we shall use later to prove the main results in this
paper. We begin by the following proposition
\begin{proposition}
\label{p1}For every  object
$\big((T(V),\partial^{(n)}),b_{n+1}\big)$ in  $\Bbb
{AD}_{n}^{n+1}$, there exists an object  $(T(V_{\leq
n+1}),\partial^{(n+1)})$ in $\mathbf {PDGA}_{n+1}$  such that:
\begin{eqnarray}
F_{n}^{n+1}((T(V_{\leq
n+1}),\partial^{(n+1)}))&=&\big((T(V),\partial^{(n)}),b_{n+1}\big)\label{12}.
\end{eqnarray}
\end{proposition}
\begin{proof} For the given homomorphism $b_{n+1}$  corresponds a homomorphism
$\mu _{n+1}$ making the following diagram commutes:

\begin{picture}(300,85)(-100,0)
\put(-5,60){$\ker
d_{n+1}\hspace{1mm}\vector(1,0){75}\hspace{1mm}Z_{n}(T(V_{\leq
n}))$}
 \put(70,65){$\mu_{n+1}$}
\put(135,10){$\Gamma^{T(V)}_{n}\longrightarrow
\mathrm{Coker}\,b_{n+1}\to 0$} \put(65,16){$b_{n+1}$}
 \put(138,55){\vector(0,-1){34}}
\put(-20,10){$H_{n+1}(V,d)$} \put(11,57){\vector(0,-1){35}}
\put(35,13){$\vector(1,0){95}$}
\end{picture}

\noindent where $Z_{n}(T(V_{\leq n}))$ is the sub-module of the
$n$-cycles of $T(V_{\leq n}).$  Let $(z_{n+1,\sigma})_{\sigma\in
\Sigma}$ and $(l_{n+1,\sigma})_{\sigma'\in \Sigma'}$ be two chosen
bases of $\ker d_{n+1}$ and $(\mathrm{Im}\,d'_{n+1})$
respectively. Since from the relations (\ref{d1}) and (\ref{22})
we have:
$$\beta_{n}\circ d_{n+1}(l_{n+1,\sigma})=\partial_{n}
^{(n)}(
d_{n+1}(l_{n+1,\sigma}))+\mathrm{Im}\,\partial^{(n)}_{n+1}=0
\text{ in }H_{n}(T(V_{\leq n}))$$ it results that  the $n$-cycle
$\partial_{n} ^{(n)}( d_{n+1}(l_{n+1,\sigma}))$ is a $n$-boundary,
so there exists  an element $a_{n+1,\sigma}\in T_{n+1}(V_{\leq
n})$ satisfying:
\begin{eqnarray}
 \partial_{n}
^{(n)}(d_{n+1}(l_{n+1,\sigma}))=\partial_{n+1}
^{(n)}(a_{n+1,\sigma})\label{27}.
\end{eqnarray}

\noindent Thus by using the decomposition of  $V_{n+1}$ given in
(\ref{a1}) we define $\partial^{(n+1)} $ on  $T(V_{\leq n+1})$ by
setting:
\begin{eqnarray}
\partial^{(n+1)} _{n+1}(z_{n+1,\sigma}+l_{n+1,\sigma})&=&\mu_{n+1}(z_{n+1,\sigma})+d_{n+1}(l_{n+1,\sigma})-a_{n+1,\sigma}\label{a24}\\
\partial^{(n+1)}_{\leq n} &=&\partial _{\leq n}^{(n)}\nonumber,
\end{eqnarray}

\noindent Now by the above diagram  we conclude that the element
$\mu_{n+1}(z_{n+1,\sigma})\in Z_{n}(T(V_{\leq n}))$ so $\partial
^{(n+1)} _{n}(\mu_{n+1}(z_{n+1,\sigma}))=0$ and according to
relation (\ref{a24}) we have:
 \begin{eqnarray*}
 \partial ^{(n+1)} _{n}( d_{n+1}(l_{n+1,\sigma})-a_{n+1,\sigma})=\partial ^{(n)} _{n}(
 d_{n+1}(l_{n+1,\sigma})-a_{n+1,\sigma})=0,
\end{eqnarray*}
therefore  $\partial _{n}^{(n+1)}\circ \partial _{n+1}^{(n+1)}=0$
and $\partial^{(n+1)}$ is a differential. On the one hand by its
construction it is easy to verify that  the Whitehead exact
sequence associated with $(T(V_{\leq n+1}),\partial^{(n+1)})$ can
be written:
\begin{equation*}
\cdots \rightarrow H_{n+1}(V,d)\overset{b_{n+1}}{\longrightarrow
}\Gamma ^{T(V_{\leq n+1})}_{n}\rightarrow H_{n}(T(V_{\leq
n+1}))\rightarrow\cdots
\end{equation*}
and on the other hand  by going back to remark \ref{20} we have
seen that if the $n$-cycle $\partial_{n+1}(l_{n+1,\sigma ^{\prime
}})$ is written as $\partial_{n+1}(l_{n+1,\sigma ^{\prime
}})=d_{n+1}(l_{n+1,\sigma '})+q_{l_{n+1,\sigma
'}}+\widetilde{q}_{l_{n+1,\sigma '}}$, then  the homomorphism
$\varphi _{n}$ is defined by $\varphi _{n}(l_{n+1,\sigma ^{\prime
}})=(\widetilde{q}_{l_{n+1,\sigma '}}+
\mathrm{Im}\,\partial_{n+1}+\mathrm{Im}\, \phi _{n+1})$. But the
relation (\ref{a24}) implies that $q_{l_{n+1,\sigma
'}}=-a_{n+1,\sigma'}$ and the element
$\widetilde{q}_{l_{n+1,\sigma '}}$ is nil, then $\varphi _{n}$ is
also nil, so $(T(V_{\leq n+1}),\partial ^{(n+1)} )$ is
$n$-perfect. Now since, by hypothesis  $(T(V_{\leq n}),\partial
^{(n)} )$ is perfect, we deduce tha $(T(V_{\leq n+1}),\partial
^{(n+1)} )$ is also perfect. Finally from the definition of the
functor $F^{n+1}_{n}$ we easy get the
  relation (\ref{12})
\end{proof}

\begin{proposition}
\label{p2} Let $(T(V),\partial)$, $(T(W),\delta)$ two perfect dgas
and let:  $$(\xi _{*},\alpha ):\big((T(V),\partial
^{(n)}),b_{n+1}\big)\rightarrow \big((T(W),\delta
^{(n)}),b'_{n+1}\big)$$ be a  morphism between their  respective
adapted system of order
 $n$. Then there exists  a adg-morphism $\chi
:(T(V_{\leq n+1}),\partial^{(n+1)})\to (T(W_{\leq
n+1}),\delta^{(n+1)})$ satisfying:
\begin{equation}
H_{\leq n+1}(\overline{\chi})=H_{\leq n+1}(\xi _{*}).\label{X13}
\end{equation}
Moreover  $(\xi _{*},\alpha )$ is an equivalence if and only if
 $\chi$ is a quasi-isomorphism.
\end{proposition}
\begin{proof}
First consider the following diagram:

\begin{picture}(300,160)(10,-30)
\put(60,100){$V_{n+1}\hspace{1mm}\vector(1,0){160}\hspace{1mm}W_{n+1}$}
 \put(50,76){$\beta_{n+1}$} \put(268,76){$\beta'_{n+1}$}
\put(76,96){$\vector(0,-1){38}$} \put(265,96){$\vector(0,-1){38}$}
\put(155,103){$\xi_{n+1}$} \put(165,52){$H_{n}(\alpha)$}
\put(60,48){$H_{n}(T(V_{\leq n}))
\hspace{1mm}\vector(1,0){120}\hspace{1mm}H_{n}(T(W_{\leq
n}))\vector(1,0){40}\hspace{1mm}W_{n}$}\put(325,52){$j'_{n}$}
\put(60,-10){$\Gamma_{n}^{T(V)}\oplus \ker
\beta_{n}\hspace{1mm}\vector(1,0){110}\hspace{1mm}\Gamma_{n}^{T(W)}\oplus
\ker
\beta'_{n}$}\put(76,45){$\vector(0,-1){38}$}\put(57,20){$i\cong$}
\put(265,46){$\vector(0,-1){38}$} \put(245,20){$i'\cong$}
\put(165,-6){$\gamma_{n}^{\alpha}\oplus \xi_{n}$}
\end{picture}

\noindent Where the lower square commutes by the relation
(\ref{a31}). Since $j'_{n}\circ(H_{n}(\alpha)\circ \beta
_{n+1}-\beta' _{n+1}\circ \xi_{n+1})=0$, so
 $\mathrm{Im}\,(H_{n}(\alpha)\circ \beta
_{n+1}-\beta' _{n+1}\circ \xi_{n+1})\subseteq \Gamma_{n}^{T(W)}$.
Also since  the dgas $(T(V),\partial)$ and  $(T(W),\delta)$ are
perfect, from the formula  (\ref{a13})  we have:
\begin{eqnarray}
i\circ \beta _{n+1}=b_{n+1}\circ pr\oplus d_{n+1}
\hspace{2cm}i'\circ \beta' _{n+1}=b'_{n+1}\circ pr'\oplus d'_{n+1}
\label{34}
\end{eqnarray}
Let us decompose $V_{n+1}\cong \ker d_{n+1}\oplus
\mathrm{Im}\,d_{n+1}$ as in (\ref{a1}). If  $(z_{n+1,\sigma
})_{\sigma \in \sum }$ and $(l_{n+1,\sigma '})_{\sigma'\in \sum'}$ are  respectively two chosen  bases of $%
\ker d_{n+1}$ and $(\mathrm{Im}\,d_{n+1})'$, then on one hand
 we get:
\begin{eqnarray}
i'\circ  H_{n}(\alpha)\circ  \beta _{n+1}(z_{n+1,\sigma}\oplus
l_{n+1,\sigma'})&=&(\gamma_{n}^{\alpha}\oplus \xi_{n})\circ
(b_{n+1}\circ pr\oplus d_{n+1})(z_{n+1,\sigma}\oplus l_{n+1,\sigma'})\nonumber\\
&=&\gamma_{n}^{\alpha}\circ b_{n+1}\circ pr(z_{n+1,\sigma})\oplus
\xi_{n}\circ d_{n+1}(l_{n+1,\sigma'})\nonumber
\end{eqnarray}
and on the other hand we have:
\begin{eqnarray}
i' \circ  \beta' _{n+1}\circ  \xi_{n+1}(z_{n+1,\sigma}\oplus
l_{n+1,\sigma'})&=& (b'_{n+1}\circ pr'\oplus d'_{n+1})(\xi_{n+1}(z_{n+1,\sigma})\oplus \xi_{n+1}(l_{n+1,\sigma'}))\nonumber\\
&=& b'_{n+1}\circ pr'\circ \xi_{n+1}(z_{n+1,\sigma})\oplus
d'_{n+1}\circ  \xi_{n+1}(l_{n+1,\sigma'}).\nonumber
\end{eqnarray}
 Since $\xi_{*}$ is a chain map, according to the relation (\ref{103}),  we deduce
that:
$$ H_{n}(\alpha)\circ \beta _{n+1}-\beta'
_{n+1}\circ \xi_{n+1}=0$$ which implies that the element
$(H_{n}(\alpha) \beta _{n+1}-\beta' _{n+1}
\xi_{n+1})(z_{n+1,\sigma}\oplus l_{n+1,\sigma'})$ is nil in
$H_{n}(T(W_{\leq n}))$. But we know that:
\begin{eqnarray*}
(H_{n}(\alpha)\beta _{n+1}-\beta'
_{n+1}\xi_{n+1})(z_{n+1,\sigma}\oplus
l_{n+1,\sigma'})=(\alpha_{n}\partial_{n}-\delta_{n}\xi_{n+1})(z_{n+1,\sigma}\oplus
l_{n+1,\sigma'})+\mathrm{Im}\,\delta_{n+1}.
\end{eqnarray*}
The last condition means that the element
$(\alpha_{n}\partial_{n}-\delta_{n}\xi_{n+1})(z_{n+1,\sigma}\oplus
l_{n+1,\sigma'})$ belongs to $\mathrm{Im}\,\delta_{n+1}$, so there
exists an element
 $y_{n+1,\sigma}\in T_{n+1}(W_{\leq n})$ satisfying:
\begin{eqnarray}
 (\alpha_{n}\partial_{n}-\delta_{n}\xi_{n+1})(z_{n+1,\sigma}\oplus
l_{n+1,\sigma'}))=\partial_{n+1}
^{(n)}(y_{n+1,\sigma})\label{a27}.
\end{eqnarray}

\noindent Define $\chi$ by setting:
\begin{eqnarray}
\chi_{n+1} (z_{n+1,\sigma}\oplus
l_{n+1,\sigma'})&=&\xi_{n+1}(z_{n+1,\sigma}\oplus
l_{n+1,\sigma'})-y_{n+1,\sigma}\nonumber\\
\chi_{\leq n}& =&\alpha_{\leq n}\label{104}
\end{eqnarray}
Clearly  by using the relation (\ref{a27}) we have:
\begin{eqnarray*}
\delta_{n}\circ \chi_{n+1}(z_{n+1,\sigma}\oplus l_{n+1,\sigma'})
=\delta_{n}(\xi_{n+1}(z_{n+1,\sigma}\oplus
l_{n+1,\sigma'})-\delta_{n}(y_{n+1,\sigma}))
=\alpha_{n}\partial_{n}=\xi_{n}\partial_{n}
\end{eqnarray*}
so  $\chi$ is a dga morphism. Finally since the element
$y_{n+1,\sigma}\in T_{n+1}(W_{\leq n})$, from the formula we
deduce that the chain map $\widetilde{\chi}_{\leq n+1}:(V_{\leq
n+1},d)\to (W_{\leq n+1},d')$, induced by $\chi$ on the
indecomposables, coincides with the chain map   $\xi_{*}$, hence
(\ref{X13}) is satisfied
\end{proof}

Now from the last proposition  we derive the following theorem
which constitute  the  first main  result in this work.
\begin{theorem}
\label{X9} Two  perfect  dgas are quasi-isomorphic  if and only if
their Whitehead exact sequences are isomorphic.
\end{theorem}
\begin{proof}
Let $(T(V),\partial)$ and $(T(W),\delta)$ be two perfect  dgas
such that their Whitehead exact sequences are isomorphic, e.i; we
have the following commutative diagram:

\begin{picture}(300,90)(25,10)
\put(65,80){$\ldots \rightarrow
H_{n+1}(V,d)\overset{b_{n+1}}{\longrightarrow }\Gamma^{T(V)}
_{n}\longrightarrow H_{n}(T(V))\longrightarrow
H_{n}(V,d)\overset{b_{n}}{\longrightarrow }...$}
\put(245,76){$\vector(0,-1){46}$}
\put(186,76){$\vector(0,-1){46}$}
 \put(65,20){$\ldots \rightarrow
H_{n+1}(W,d')\overset{b'_{n+1}}{\longrightarrow }\Gamma^{T(W)}
_{n}\longrightarrow H_{n}(T(W))\longrightarrow
H_{n}(W,d')\overset{b'_{n}}{\longrightarrow }...$}
\put(110,76){$\vector(0,-1){46}$} \put(110,50){ $f_{n+1}$}
\put(330,50){ $f_{n}$} \put(330,76){$\vector(0,-1){46}$}
\put(187,50){ $\gamma_{n}$} \put(245,50){ $h_{n}$} \put(0,50){
$(B)$}
 \put(95,50){ $\cong$}\put(230,50){ $\cong$}
 \put(315,50){ $\cong$}\put(170,50){ $\cong$}
\end{picture}

\noindent First by the homotopy extension theorem
 \cite{MAL}, the given graded isomorphism
$f_{*}:H_{*}(V,d)\rightarrow H_{*}(W,d')$ yields  a chain map
$\xi_{*}:(V,d)\to (W,d')$ such that $H_{*}(\xi_{*})=f_{*}$.

\noindent  Assume, by induction, that we have constructed  a map
$\alpha^{n} :(T(V_{\leq n+1}),\partial^{(n+1)})\to (T(W_{\leq
n+1}),\delta^{(n+1)})$ such that $H_{\leq
n-1}(\widetilde{\alpha}^{n})=H_{\leq n-1}(\xi_{*})=f_{\leq n-1}$.
Now consider the following pair:
 $$(\xi_{*},\alpha^{n}):F_{n}(T(V),\partial)=\big((T(V),\partial
^{(n)}),b_{n+1}\big)\rightarrow \big((T(W),\delta
^{(n)}),b'_{n+1}\big)=F_{n}(T(W),\delta)$$

 \noindent By
the commutativity of the diagram $(B)$ we deduce that the above
pair is a morphism in $\Bbb{AD}_{n}$. Then  proposition \ref{p2}
provides   a chain map  $\alpha^{n+1}:(T(V_{\leq
n+1}),\partial^{(n+1)})\to (T(W_{\leq n+1}),\delta^{(n+1)})$
 satisfying $H_{\leq
n}(\widetilde{\alpha}^{n+1} )=H_{\leq n}(\xi_{*})=f_{\leq n}$.
Again by the commutativity of the diagram $(B)$ we deduce that the
pair:
$$(\xi_{*},\alpha^{n+1}):\big((T(V),\partial
^{(n+1)}),b_{n+2}\big)\rightarrow \big((T(W),\delta
^{(n+1)}),b'_{n+2}\big)$$
  is a morphism in $\Bbb{AD}_{n+1}$. Therefore   proposition \ref{p2}
provides a chain map  $\alpha^{n+2}:(T(V_{\leq
n+2}),\partial^{(n+2)})\to (T(W_{\leq n+2}),\delta^{(n+2)})$
satisfying $H_{\leq n+1}(\widetilde{\alpha}^{n+2} )=H_{\leq
n+1}(\xi_{*})=f_{\leq n+1}$. If we continue this iterative
construction  we will find a chain  map $\alpha:(T(V),\partial)\to
(T(W),\delta)$ satisfying the relation $H_{n}(\widetilde{\alpha}
)=f_{n}$ for all $n$. Finally since $f_{*}$ is an isomorphism,
then so is $H_{*}(\widetilde{\alpha})$ and Moore's theorem
\cite{M1} asserts that $\alpha$ is a quasi-isomorphism. For the
converse it is well-known that the Whitehead exact sequence is an
invariant of quasi-isomorphism.
\end{proof}
We can summarize propositions \ref{p1} and \ref{p2} in the
following theorem:
\begin{theorem}
\label{1t} The Functor $F_{n}^{n+1}$ provides a bijection between
the set of the quasi-isomorphism types in the category
$\mathbf{PDGA}_{n+1}$ and the set of isomorphism classes in the
category $\Bbb {AD}_{n}$

\end{theorem}
\subsection{Algorithm}
Let  $H_{*}$ be a graded $R$-module such that $H_{i}=0$ for $i=0$.
Denote by $S(H_{*})$ the set of quasi-isomorphism types of perfect
dgas $(T(V),\partial)$ satisfying $H_{*}(V,d)=H_{*}$. In this
section we shall use the results already obtained in the previous
paragraph
                                s in order to describe an algorithm that allows us to compute the
cardinal of
 $S(H_{*}).$

Indeed; First  choose   a free chain complex   $%
(C_{\ast },d)$ such that  $H_{\ast }=H_{\ast }(C_{\ast },d)$ and
consider $(T(V_{\leq 2}),d)$. Clearly $(T(V_{\leq 2}),d)$ is
perfect.

 An adapted  system is a family of homomorphisms $(b_{n})_{n\geq 3}$
defined in the following recursive manner:

\noindent  $b_{3}$ is a homomorphism $H_{3}\to
\Gamma_{2}^{T(V_{\leq 2})}$ so that  the pair $\big((T(V),\partial
^{(2)}),b_{3}\big)$  is an adapted  system of order 3. So by
proposition \ref{p1} we can get
 a perfect  dga  $(T(V_{\leq 3}),\partial^{(3)})$. Next
$b_{4}$ is a homomorphism $H_{4}\to \Gamma_{3}^{T(V_{\leq 3})}$ so
that the pair $\big((T(V),\partial ^{(3)}),b_{4}\big)$ is an
adapted system of order 4 and   proposition \ref{p1} provides
 a perfect  dga  $(T(V_{\leq 4}),\partial^{(4)})$. Assume
now that we have constructed  a perfect dga $(T(V_{\leq
n}),\partial^{(n)})$ by the process describe as above, then
$b_{n+1}$ is a homomorphism $H_{n+1}\to \Gamma_{n}^{T(V_{\leq
n})}$ which implies that the pair $\big((T(V),\partial
^{(n)}),b_{n+1}\big)$ is an adapted  system of order
 $n$. If we iterate this process we find a perfect dga $(T(V),\partial)$ satisfying for all $n$ the relation
$F_{n}(T(V),\partial)=\big((T(V),\partial ^{(n)}),b_{n+1}\big)$
where $F_{n}$ is the functor defined in (\ref{15}).

\noindent We  say that two adapted  systems $(b_{n})_{n\geq 3},
(b_{n}')_{n\geq 3}$ are equivalent, and we write $(b_{n})_{n\geq
3}\sim (b_{n}')_{n\geq 3}$, if there  exists an graded isomorphism
$f_{*}:H_{*}\to H_{*}$ making the following diagram commutes for
all $n\geq 2$:

\begin{picture}(300,80)(-80,-28)
\put(134,26){$\vector(0,-1){28}$} \put(57,27){$\vector(0,-1){28}$}
\put(50,-12){$H_{n+1}\hspace{1mm}\vector(1,0){50}\hspace{1mm}
\Gamma_{n}^{T(V))} $} \put(35,9){$f_{n+1}$} \put(136,9){$\cong$}
\put(50,30){$H_{n+1}\hspace{1mm}\vector(1,0){50}\hspace{1mm}
\Gamma_{n}^{T(V)}$} \put(85,34){$b_{n+1}$}
\put(85,-08){$b'_{n+1}$}
\end{picture}

\noindent Let  $[(b_{n})_{n\geq 3}]$ the  an equivalence class of
the family $(b_{n})_{n\geq 3}$.  From theorem \ref{1t} we deduce
that the Cardinal of the set $S(H_{*})$ is equal to the number of
such equivalence classes.

Now we shall illustrate  with examples how we can compute the
Cardinal of $S(H_{*})$ by using the algorithm described above.
\begin{example}
\label{25} -  If $H_{\leq n}$ is a graded abelian group such that
$H_{1}=H_{3}=\Bbb Z$, then   $S(H_{\leq n})$ is infinite. This
observation comes from  the fact that in this case we have an
infinity of homomorphisms $b_{3}\in Hom(H_{3},H_{1}\otimes
H_{1})=Hom(\Bbb Z,\Bbb Z)=\Bbb Z$ and  if $b_{3}\neq b'_{3}$, then
$(b_{n})_{n\geq 3}\nsim (b'_{n})_{n\geq 3}$. Recall that here we
have used that $\Gamma_{2}^{T(V)}=H_{1}\otimes H_{1}$. We conclude
that there exists an infinite classes $[(b_{n})_{n\geq 3}]$.
Therefore the set $S(H_{\leq n})$ is infinite for all $n\geq 4$.

-  If  graded module  $H_{\leq n}$ is such that  $H_{i}$ is finite
or all $i\leq n$, then  we deduce  that the  set  $S(H_{\leq n})$
is also finite.

- Let $H_{\leq 10}$ be an abelian graded group such that:
$$\begin{array}{lcrlcrlcrlcrlcrlcr}
H_{3}&=&H_{7}&=&\mathbb{Z}_{2},&&H_{4}&=&\mathbb{Z}_{3},&&H_{8}&=&\mathbb{Z}_{4}\\
H_{6}&=&H_{5}&=&\mathbb{Z}_{5},&&H_{9}&=&\mathbb{Z}_{3},&&H_{10}&=&\mathbb{Z},\hspace{10mm}H_{i}=0
\text{ otherwise}
\end{array}$$
Then Card  $S(H_{*})=18.$ Indeed; on the one hand by the above
algorithm we know that   Card $S(H_{*})$ is equal to the number of
equivalence  classes
$[(b_{10},b_{9},b_{8},b_{7},b_{6},b_{5},b_{4},b_{3})]$, where
$b_{i}\in Hom(H_{i},\Gamma_{i-1})$. On the other hand in
(\cite{ben3}, th.13 p. 127) il is shown that:
\begin{eqnarray*}
\Gamma_{6}&=&H_{3}\otimes H_{3}\,\,\,,\hspace{1cm}\Gamma_{i}=0\,\,,\,\,\forall i\leq 5\\
\Gamma_{7}&=&H_{3}\otimes H_{4}\oplus H_{4}\otimes H_{3}\oplus Tor(H_{3}, H_{3})\\
\Gamma_{8}&=&H_{3}\otimes H_{5}\oplus H_{4}\otimes H_{4}\oplus
H_{5}\otimes H_{3}\oplus Tor(H_{3}, H_{4})\oplus
Tor(H_{4}, H_{3})\\
\Gamma_{9}&=&H_{3}\otimes H_{6}\oplus H_{4}\otimes H_{5}\oplus
H_{5}\otimes H_{4}\oplus H_{3}\otimes H_{6}\oplus
Tor(H_{3}, H_{5})\oplus\\
&& Tor(H_{4}, H_{4})\oplus Tor(H_{5}, H_{3})\oplus
\frac{H_{3}\otimes H_{3}\otimes H_{3}}{\mathrm{Im}\,b_{7}\otimes
H_{3}+H_{3}\otimes \mathrm{Im}\,b_{7}}
\end{eqnarray*}
where $b_{7}\in Hom(H_{7},\Gamma_{6})$, therefore according to the
hypothesis we derive:
\begin{eqnarray*}
\Gamma_{6}&=&\Bbb
Z_{2}\hspace{1cm}\Gamma_{7}=\mathbb{Z}_{2}\hspace{1cm}\Gamma_{8}=\mathbb{Z}_{3}
\end{eqnarray*}
Now since  $\Gamma_{i}=0\,,\,\forall i\leq 5$ we deduce that
$b_{i}\,,\,\forall i\leq 6.$ So we shall compute  the adapted
systems on form $(b_{10},b_{9},b_{8},b_{7},0,0,0,0)$. A simple
computation  shows that
$Hom(H_{7},\Gamma_{6})=Hom(\mathbb{Z}_{2},\mathbb{Z}_{2})=\mathbb{Z}_{2},$
it results from this two  homomorphisms $b^{(1)}_{7}=1$ and
$b^{(0)}_{7}=0$ for which  correspond two  adapted systems
 $(b_{10},b_{9},b_{8},0,0,0,0,0)$,
$(b_{10},b_{9},b_{8},1,0,0,0,0)$.

\noindent Next we have
$Hom(H_{8},\Gamma_{7})=Hom(\mathbb{Z}_{4},\mathbb{Z}_{2})=\mathbb{Z}_{2}$,
so we find   two homomorphisms $b^{(1)}_{8}=1$ et $b^{(0)}_{8}=0$
for which correspond 4 adapted systems which are:
$$
\begin{array}{lcrlcrlcrlcrlcrlcr}
(b_{10},b_{9},0,0,0,0,0,0)&\hspace{1cm}&(b_{10},b_{9},1,0,0,0,0,0)\\
(b_{10},b_{9},0,1,0,0,0,0)&\hspace{1cm}&(b_{10},b_{9},1,1,0,0,0,0).
\end{array}
$$
Next  we have
$Hom(H_{9},\Gamma_{8})=Hom(\mathbb{Z}_{6},\mathbb{Z}_{3})=\mathbb{Z}_{3}$
so we get 3 homomorphisms $b^{(0)}_{9}=1$, $b^{(1)}_{9}=0$ and
$b^{(2)}_{9}=0$ for which  correspond 12 adapted systems which
are:
$$
\begin{array}{lcrlcrlcrlcrlcrlcr}
(b_{10},0,0,0,0,0,0,0)&\hspace{1cm}&(b_{10},1,0,0,0,0,0,0)&\hspace{1cm}&(b_{10},2,0,0,0,0,0,0)\\
(b_{10},0,1,0,0,0,0,0)&\hspace{1cm}&(b_{10},1,1,0,0,0,0,0)&\hspace{1cm}&(b_{10},2,1,0,0,0,0,0)\\
(b_{10},0,0,1,0,0,0,0)&\hspace{1cm}&(b_{10},1,0,1,0,0,0,0)&\hspace{1cm}&(b_{10},2,0,1,0,0,0,0)\\
(b_{10},0,1,1,0,0,0,0)&\hspace{1cm}&(b_{10},1,1,1,0,0,0,0)&\hspace{1cm}&(b_{10},2,1,1,0,0,0,0)
\end{array}
$$
Now we shall determine $Hom(H_{10},\Gamma_{9})$. Since
$\Gamma_{9}$ depends on the homomorphism $b_{7}$
 we deduce  that we have two abelian
 groups
$\Gamma^{b^{(0)}_{7}}_{9}=\mathbb{Z}_{2}$ et
$\Gamma^{b^{(1)}_{7}}_{9}=0$,  so we get:
$$
\begin{array}{lcrlcrlcrlcrlcrlcr}
Hom(H_{10},\Gamma^{b^{(0)}_{7}}_{9})&=&Hom(\mathbb{Z},\mathbb{Z}_{2})&=&\mathbb{Z}\\
Hom(H_{10},\Gamma^{b^{(1)}_{7}}_{9})&=&Hom(\mathbb{Z},0)&=&0
\end{array}
$$
Hence from the first  case results two  homomorphisms
$b^{(0)}_{10}=1$, $b^{(1)}_{10}=0$ for which  correspond 12
adapted systems which are:
$$
\begin{array}{lcrlcrlcrlcrlcrlcr}
(0,0,0,0,0,0,0,0)&\hspace{-3mm}&(1,0,0,0,0,0,0,0)&\hspace{-3mm}&(0,0,1,0,0,0,0,0)&\hspace{-3mm}&(1,0,1,0,0,0,0,0)\\
(0,1,0,0,0,0,0,0)&\hspace{-3mm}&(1,1,0,0,0,0,0,0)&\hspace{-3mm}&(0,1,1,0,0,0,0,0)&\hspace{-3mm}&(1,1,1,0,0,0,0,0)\\
(0,2,0,0,0,0,0,0)&\hspace{-3mm}&(1,2,0,0,0,0,0,0)&\hspace{-3mm}&(0,2,1,0,0,0,0,0)&\hspace{-3mm}&(1,2,1,0,0,0,0,0)
\end{array}
$$
and from the second case we only derive  the nil homomorphism for
which  corresponds 6 adapted systems which are:
$$
\begin{array}{lcrlcrlcrlcrlcrlcr}
(0,0,0,1,0,0,0,0)&\hspace{1cm}&(0,1,0,1,0,0,0,0)&\hspace{1cm}&(0,1,1,1,0,0,0,0)\\
(0,2,0,1,0,0,0,0)&\hspace{1cm}&(0,0,1,1,0,0,0,0)&\hspace{1cm}&(0,2,1,1,0,0,0,0)
\end{array}
$$
In total we  get  18 adapted systems which are obviously not
equivalent,  so we deduce that  Card $S(H_{\leq 10})=18$.
\end{example}
\section{General case }
 Our  aim now is to generalize the precedent results to any free dga  ( not necessary perfect ).  Let $(T(V),\partial)$ be a simply connected CW-complex $X$ and let:
\begin{equation*}
\cdots \rightarrow H_{n+1}(V,d)\overset{b_{n+1}}{\longrightarrow }%
\Gamma ^{T(V)}_{n}\rightarrow H_{n}(T(V))\rightarrow H_{n}(V,d)\overset{%
b_{n}}{\longrightarrow }\cdots
\end{equation*}
its Whitehead exact sequence. We start this section by proving the
following
 proposition:
\begin{proposition}
\label{2} To every free dga  $(T(V),\partial)$ corresponds a pair
$\big((T(V),\widetilde{\partial}),(\pi_{n})_{n\geq 2}\big)$
consisting with a perfect dga $(T(V),\widetilde{\partial})$ and a
family of extensions $(\pi_{n})_{n\geq 2}$, where $\pi_{n}\in
\mathrm{Ext}(H_{n}(V,d),\mathrm{Coker}\,b_{n+1}).$
\end{proposition}
\begin{proof}
The Whitehead exact sequence associated with $(T(V),\partial)$ can
split in order to give  a family of  homomorphisms
$\{b_{n+1}:H_{n+1}(V,d)\to \Gamma^{T(V)} _{n}\}_{n\geq 2}$ and a
family of extensions $\{[\mathrm{Coker}\,b_{n+1}\rightarrowtail
H_{n}(T(V))\twoheadrightarrow \ker b_{n}]\}_{n\geq 2}$.

\noindent But according to  remark \ref{20} we know that for every
short exact sequence  $\mathrm{Coker}\,b_{n+1}\rightarrowtail
H_{n}(T(V))\twoheadrightarrow \ker b_{n}$ corresponds a homomorphism $\overline{\varphi _{n}}:(\mathrm{Im}\,d_{n+1})'\overset{\varphi _{n}%
}{\rightarrow }\Gamma^{T(V)} _{n}\twoheadrightarrow \mathrm{Coker}%
\,b_{n+1}$ and by using  the resolution of $H_{n}(V,d)$ given in
(\ref{19}) the homomorphism $\overline{\varphi _{n}}$
defines an extension $\pi_{n}\in Ext_{R}(H_{n}(V,d),\mathrm{Coker}%
\,b_{n+1}).$

\noindent Finally suppose that we have constructed, by induction,
the perfect  dga $(T(V_{\leq n}),\widetilde{\partial}^{(n)})$
which is the sub-dga   of $(T(V),\widetilde{\partial})$. Then the
homomorphism $b_{n+1}:H_{n+1}(V,d)\to \Gamma^{T(V)} _{n}$ implies
that $\big((T(V),\widetilde{\partial} ^{(n)}),b_{n+1}\big)$
 is an adapted  system of order $n$, so by proposition 2 we
get a perfect  dga  $(T(V_{\leq
n+1}),\widetilde{\partial}^{(n+1)})$ etc...
\end{proof}
\begin{definition}
\label{38}  $\big((T(V),\widetilde{\partial}),(\pi_{n})_{n\geq
2}\big)$ is called a characteristic pair  of  $(T(V),\partial)$.
\end{definition}
\begin{remark}
\label{42} It is important to note that the graded homomorphism
$b_{*}$ which appears  in the Whitehead sequence associated with
$(T(V),\widetilde{\partial})$ is equal to the one's which appears
in  the Whitehead sequence associated with $(T(V),\partial)$. The
difference between the two these sequences is that  the first one
splits, e.i; the extension $\mathrm{Coker}\,b_{n+1}\rightarrowtail
H_{n}((T(V),\widetilde{\partial}))\twoheadrightarrow \ker b_{n}$
is trivial for all $n\geq 2$, while the second one is not,  e.i;
the extension $\mathrm{Coker}\,b_{n+1}\rightarrowtail
H_{n}((T(V),\partial))\twoheadrightarrow \ker b_{n}$ is not
trivial in general. It follows that if  $(T(V),\partial)$ and
$(T(W),\delta)$ are quasi-isomorphic, then their Whitehead exact
sequences are isomorphic and  so are the Whitehead exact sequences
associated with $(T(V),\widetilde{\partial})$ and
$(T(W),\widetilde{\delta})$, therefore by theorem 1 we deduce that
$(T(V),\widetilde{\partial})$ and $(T(W),\widetilde{\delta})$ are
quasi-isomorphic. Hence the perfect dga associated with a free dga
is unique, up to quasi-isomorphism.
\end{remark}

A natural question arises from the  remark \ref{42}: if
$(T(V),\widetilde{\partial})$ and $(T(W),\widetilde{\delta})$ are
quasi-isomorphic  what condition  should-we add in order to have a
quasi-isomorphism between $(T(V),\partial)$ and $(T(W),\delta)$?
The following paragraph is devoted to answer this question.
\subsection{The category $\mathbf{C}$}
In this paragraph we define a new  category $\mathbf{C}$ as
follows:

\noindent \textbf{Object}: is a pair
$\big((T(V),\widetilde{\partial}),(\pi_{n})_{n\geq 2}\big)$ where
$(T(V),\widetilde{\partial})$ is a perfect  dga and where
$\pi_{n}\in
Ext(H_{n}(V,d),\mathrm{Coker}\,b_{n+1})$ for all $n\geq 2$. Recall that the graded homomorphism $b_{*}$ is given by the Whitehead exact sequence associated with $(T(V),\widetilde{\partial})$.\\

\noindent \textbf{Morphism}: a morphism
$\alpha:\big((T(V),\widetilde{\partial}),(\pi_{n})_{n\geq
2}\big)\rightarrow
\big((T(W),\widetilde{\delta}),(\pi'_{n})_{n\geq 2}\big)$ between
two objects in $\mathbf{C}$ is a dga morphism
$\alpha:(T(V),\widetilde{\partial})\rightarrow
(T(W),\widetilde{\delta})$ satisfying:
\begin{equation}
(H_{n}(\alpha)^{\ast })(\pi' _{n})=(\gamma _{n}^{\alpha})_{\ast
}(\pi _{n}),\,\,\,\,\forall n\geq 2\label{47}
\end{equation}
where the homomorphism $\widetilde{\gamma} _{n}^{\alpha}$ is
induced by $H _{n}(\alpha)$ on the quotient group
$\mathrm{Coker}\,b_{n+1}$ and where $(H_{n}(\alpha))^{\ast }$ ,
$(\widetilde{\gamma} _{n}^{\alpha})_{\ast }$ are such that:
\begin{eqnarray*}
(\widetilde{\gamma }_{n}^{\alpha})_{\ast }
&:&\mathrm{Ext}(H_{n}(V,d),\mathrm{Coker}\,b_{n+1})\rightarrow
\mathrm{Ext}(H_{n}(V,d),\mathrm{Coker}
\,b'_{n+1}) \\
(H_{n}(\alpha))^{\ast }
&:&\mathrm{Ext}(H_{n}(W,d'),\mathrm{Coker}\,b'_{n+1})\rightarrow
\mathrm{Ext}(H_{n}(V,d),\mathrm{Coker}\,b'_{n+1})
\end{eqnarray*}
\begin{remark}
\label{A2} The  relation (\ref{47}) means the following:

 \noindent Choose the two resolution  $(\mathrm{Im}\,d_{n+1})^{\prime }\overset{d_{n+1}}{\rightarrow
}\ker d_{n}\twoheadrightarrow H_{n}(V,d)$ and
$(\mathrm{Im}\,d'_{n+1})'\overset{d'_{n+1}}{%
\rightarrow }\ker d'_{n}\twoheadrightarrow H_{n}(W,d')$   of
$H_{n}(V,d)$ and $H_{n}(W,d')$ given in (\ref{19}). For the given
extensions $\pi _{n}$ and $\pi' _{n}$ correspond the following
commutative two diagrams:

\begin{picture}(300,130)(-30,0)

\put(20,98){$\vector(0,-1){39}$} \put(-05,102){$
(\mathrm{Im}\,d_{n+1})'\overset{d_{n+1}}{\rightarrowtail }\ker
d_{n}\twoheadrightarrow H_{n}(V,d)$}
\put(210,50){$(\mathrm{Im}\,d'_{n+1})'\overset{d'_{n+1}}{\rightarrowtail
}\ker d'_{n}\twoheadrightarrow H_{n}(W,d')$}
\put(210,102){$(\mathrm{Im}\,d_{n+1})'\overset{d_{n+1}}{\rightarrowtail
}\ker d_{n}\twoheadrightarrow H_{n}(V,d)$}
\put(0,-01){$\mathrm{Coker}\,b'_{n+1}\underset{pr'}{\twoheadleftarrow}
\Gamma_{n}^{T(W)}$} \put(0,50){$\mathrm{Coker}
\,b_{n+1}\underset{pr}{\twoheadleftarrow} \Gamma_{n}^{T(V)}$}
\put(6,78){$\widetilde{\varphi} _{n}$}
\put(20,47){$\vector(0,-1){39}$}\put(56,78){$\varphi _{n}$}
\put(208,-01){$\mathrm{Coker}
\,b'_{n+1}\underset{pr'}{\twoheadleftarrow} \Gamma_{n}^{T(W)}$}
\put(225,98){$\vector(0,-1){39}$} \put(76,47){$\vector(0,-1){38}$}
\put(225,47){$\vector(0,-1){39}$}
 \put(6,25){$\widetilde{\gamma}_{n}^{\alpha}$}
 \put(78,25){$\gamma_{n}^{\alpha}$}
 \put(209,25){$\widetilde{\varphi'}_{n}$} \put(266,25){$\varphi'_{n}$}
\put(203,78){$\xi_{n+1}$}
\put(349,78){$H_{n}(\widetilde{\alpha})$}
\put(25,97){$\vector(4,-3){48}$} \put(232,47){$\vector(4,-3){48}$}
\put(345,98){$\vector(0,-1){39}$} \put(-55,50){$(C)$}
\end{picture}

\noindent where: $$[\widetilde{\varphi} _{n}]\in Ext_{R}^{1}(H_{n}(V,d),\mathrm{Coker}%
\,b_{n+1})=\frac{Hom\big((\mathrm{Im}\,d_{n+1})^{\prime },\mathrm{Coker}%
\,b_{n+1}\big)}{(d_{n+1})^{*}\big(Hom(\ker d_{n},\mathrm{Coker}%
\,b_{n+1})\big)}$$  $$[\widetilde{\varphi'}_{n}]\in Ext_{R}^{1}(H_{n}(W,d'),\mathrm{Coker}%
\,b_{n+1})=\frac{Hom\big((\mathrm{Im}\,d_{n+1})^{\prime },\mathrm{Coker}%
\,b'_{n+1}\big)}{(d'_{n+1})^{*}\big(Hom(\ker d'_{n},\mathrm{Coker}%
\,b'_{n+1})\big)}$$ are such that $[\widetilde{\varphi
}_{n}]=\pi_{n}$ and $[\widetilde{\varphi'} _{n}]=\pi'_{n}$.

\noindent It is well-known that the homomorphisms
$(\widetilde{\gamma _{n}^{\alpha }})_{\ast }$ and $(H_{n}(\alpha))
^{\ast }$ are defined by the following formulas:
\begin{equation*}
(\widetilde{\gamma _{n}^{\alpha }})_{\ast }(\pi
_{n})=[\widetilde{\gamma _{n}^{\alpha }}\circ \varphi _{n}]
\hspace{1cm} (H_{n}(\alpha))^{\ast }(\pi' _{n})=[\varphi'
_{n}\circ \xi _{n+1}].
\end{equation*}
So  the  relation (\ref{47}) is equivalent to:
\begin{equation*}
\lbrack \widetilde{\gamma}_{n}^{\alpha}\circ \widetilde{\varphi} _{n}]=[%
\widetilde{\varphi'} _{n}\circ \xi _{n+1}]\hspace{2mm}\mathrm{in}%
\hspace{2mm}\mathrm{Ext}(H_{n}(V,d),\mathrm{Coker}\,b'_{n+1})
\end{equation*}
or  there exists a homomorphism $h_{n} :\ker d_{n} \longrightarrow
\mathrm{Coker} \,b'_{n+1}$ such that:
\begin{equation}
\widetilde{\gamma} _{n}^{\alpha }\circ \widetilde{\varphi}
_{n}-\widetilde{\varphi' }_{n}\circ \xi _{n+1}=h_{n}\circ d_{n+1}.
\label{44}
\end{equation}
Finally since   $\ker d_{n}$ is free,  $h_{n}$ induces a
homomorphism $\widetilde{h}_{n}$ making the following diagram
commutes:

\begin{picture}(300,90)(35,30)
\put(120,98){$\vector(0,-1){42}$}
\put(150,101){$\vector(3,-1){140}$}
 \put(221,80){$\widetilde{h}_{n}$}
\put(110,102){$\ker d_{n+1}$}
\put(95,45){$\mathrm{Coker}\,b_{n+1}\overset{pr}{\twoheadleftarrow}\Gamma_{n}^{T(V)}\overset{}{\twoheadleftarrow}
\Gamma_{n}^{T(V)}\oplus \ker \beta_{n}\cong H_{n}(T(V_{\leq n})$}
\put(107,78){$h_{n}$}
\end{picture}

\noindent Or on the other word, according to the diagram (C),
that:
\begin{equation}
\mathrm{Im}\,(\gamma _{n}^{\alpha }\circ \varphi _{n}-\varphi'
_{n}\circ \xi _{n+1}-\widetilde{h}_{n}\circ d_{n+1})\subset
\mathrm{Im}\,b' _{n+1}. \label{18}
\end{equation}
\end{remark}

Our aim now is to define a functor from the category $\mathbf {C}$
to the category $\mathbf {DGA}$ of the free dgas. We begin by the
following propositions.
\begin{proposition}
\label{49}  For every object
$\big((T(V),\widetilde{\partial}),(\pi_{n})_{n\geq 2}\big)$ in
$\mathbf {C}$, there exists a free dga $(T(V),\partial)$ such that
$\big((T(V),\widetilde{\partial}),(\pi_{n})_{n\geq 2}\big)$ is a
characteristic pair of $(T(V),\partial)$ (see definition
\ref{2}).
\end{proposition}
\begin{proof}
Indeed; for every  $n$  the sub-algebra $(T(V_{\leq
n}),\partial^{(n)})$ of $(T(V),\partial)$ is defined as follows:
choose $(\mathrm{Im}\,d_{n+1})^{\prime
}\overset{d_{n+1}}{\rightarrow }\ker d_{n}\twoheadrightarrow
H_{n}(V,d)$ as a
 resolution  of $H_{n}(V,d)$. The given
 extension $\pi_{n}$ provides two homomorphisms
 $\widetilde{\varphi}_{n}$, $\varphi_{n}$ and $h_{n}$  making the following diagram commutes:

\begin{picture}(300,100)(25,30)
\put(120,98){$\vector(0,-1){42}$}
\put(140,98){$\vector(3,-2){60}$} \put(170,80){$\varphi_{n}$}
\put(115,102){$(\mathrm{Im}\,d_{n+1})'\hspace{1mm}\vector(1,0){75}\hspace{1mm}Z_{n}(T(V_{\leq
n}))$}
\put(115,45){$\mathrm{Coker}\,b_{n+1}\overset{pr}{\longleftarrow}\Gamma_{n}^{T(V)}\overset{pr}{\longleftarrow}H_{n}(T(V_{\leq
n}))$} \put(107,78){$\widetilde{\varphi}_{n}$}
\put(255,98){$\vector(0,-1){42}$} \put(257,78){$pr$}
\put(197,105){$h_{n}$}
\end{picture}

\noindent Here $Z_{n}(T(V_{\leq n}))$ denotes the sub-module of
the $n$-cycles of the dga $(T(V),\widetilde{\partial})$.

\noindent First let us use the decomposition $V_{n+1}\cong \ker
d_{n+1} \oplus (\mathrm{Im}\,d_{n+1})'$ given by the relation
(\ref{28}) and let us choose $(z_{n+1,\sigma})_{\sigma\in \Sigma}$
and $(l_{n+1,\sigma'})_{\sigma'\in \Sigma'}$ as two bases of $\ker
d_{n+1}$ and $(\mathrm{Im}\,d_{n+1})'$ respectively. Now  we
define the differential  $\partial^{(n+1)}$ on $T(V_{\leq n+1})$
by setting:

$$\partial_{n+1}^{(n+1)}(z_{n+1,\sigma}+l_{n+1,\sigma'})=\widetilde{\partial}^{(n+1)}_{n+1}(z_{n+1,\sigma}+l_{n+1,\sigma'})+h_{n}(l_{n+1,\sigma'})$$
Since $\mathrm{Im}\,h_{n}\subset\,Z_{n}(T(V_{\leq n}))$ we deduce
that $\partial^{(n+1)}$ is a differential.

\noindent If we iterate the process we get a free dga
$(T(V),\partial)$ and by proposition \ref{2} it is easy to see
that $\big((T(V),\widetilde{\partial}),(\pi_{n})_{n\geq 2}\big)$
is a characteristic pair of $(T(V),\partial)$.
\end{proof}
\begin{remark}
\label{21} By combining the formulas (\ref{a24})  and (\ref{24})
we deduce that, for each $n$, the homomorphism:
$$i\circ \beta_{n+1}:V_{n+1}\cong \ker
d_{n+1} \oplus (\mathrm{Im}\,d_{n+1})'\longrightarrow
H_{n}(T(V_{\leq n}))\overset{i}{\cong}\Gamma_{n}^{T(V)}\oplus \ker
b_{n}$$ satisfies the relation:
\begin{eqnarray}
\beta_{n+1}(z_{n+1,\sigma}+l_{n+1,\sigma'})=b_{n+1}(z_{n+1,\sigma}+\mathrm{Im}\,
d_{n+1})+\varphi(l_{n+1,\sigma'})\oplus
d_{n+1}(l_{n+1,\sigma'})\label{a30}.
\end{eqnarray}
\end{remark}
 In order to pursue our study  we need the following lemma.
\begin{lemma}
\label{lem3} Let $(T(V_{\leq n+1}),\partial)$ and $(T(W_{\leq
n+1}),\delta)$ be two free dga. Assume that there exists a dga
morphism $\alpha:(T(V_{\leq n}),\partial)\rightarrow(T(W_{\leq
 n}),\delta)$ and a homomorphism $\rho_{n+1}$ making the following
diagram commutes:

\begin{picture}(300,90)(-30,30)
\put(60,100){$V_{n+1}\hspace{1mm}\vector(1,0){160}\hspace{1mm}W_{n+1}$}
 \put(49,76){$\beta_{n+1}$} \put(268,76){$\beta'_{n+1}$}
\put(76,96){$\vector(0,-1){38}$} \put(265,96){$\vector(0,-1){38}$}
\put(155,103){$\rho_{n+1}$} \put(165,52){$H_{n}(\alpha)$}
\put(60,48){$H_{n}(T(V_{\leq n}))
\hspace{1mm}\vector(1,0){100}\hspace{1mm}H_{n}(T(W_{\leq
n}))\hspace{1mm}$} \put(-29,76){$(D)$}
\end{picture}

\noindent Then we can extend $\alpha$ to a dga-morphism
$$\theta:(T(V_{n+1}),\partial)\rightarrow(T(W_{\leq n+1}),\delta)$$
\end{lemma}
\begin{proof}
First let us choose $(v_{\sigma})_{\sigma\in\Sigma}$ as a basis of
$V_{n+1}$ and recall that we have:
$$H_{n}(\alpha)\beta_{n+1}-\beta'_{n+1}\rho_{n+1}(v_{\sigma})=(\alpha_{n}\partial_{n+1}-\delta_{n+1}\rho_{n+1})(v_{\sigma})+\mathrm{Im}\,\delta_{n+1}$$

\noindent Since the diagram (D) commutes, for each $v_{\sigma}$,
the element
$(\alpha_{n}\circ\partial_{n}-\delta_{n}\circ\rho_{n+1})(v_{\sigma})\in\mathrm{Im}
\,\delta_{n+1}$, therefore there exists an element
$y_{n+1,\sigma}\in T_{n+1}(W_{\leq n})$ such that:
\begin{eqnarray}
\alpha_{n}\circ\partial_{n+1}-\delta_{n+1}\circ\rho_{n+1}(v_{\sigma})=\delta_{n+1}(y_{n+1},\sigma)\label{23}.
\end{eqnarray}
Thus we define $\theta:(T(V_{\leq
n+1}),\partial)\rightarrow(T(W_{\leq n+1}),\delta)$ by setting:
\begin{eqnarray*}
\theta_{n+1}(v_{\sigma})&=&\chi_{n+1}(v_{\sigma})-y_{n+1}\\
\theta_{n}&=&\alpha_{\leq n}
\end{eqnarray*}
Since by  (\ref{23}) we have:
$$\delta_{n+1}\circ\theta_{n+1}(v_{\sigma})=\delta_{n+1}(\xi_{n+1}(v_{\sigma}))-\delta_{n+1}(y_{n+1})=\alpha_{n}\circ\partial_{n+1}(v_{\sigma})=\theta_{n}\circ\partial_{n+1}(v_{\sigma})$$
we deduce that $\theta$ is a dga-morphism.
\end{proof}
\begin{proposition}
\label{1} Every morphism
$\alpha:\big((T(V),\widetilde{\partial}),(\pi_{n})_{n\geq
2}\big)\to \big((T(W),\widetilde{\delta}),(\pi'_{n})_{n\geq
2}\big)$ in $\mathbf {C}$ yields a map
 $\psi:(T(V),\partial)\to (T(W),\delta)$, where  $(T(V),\partial)$ and  $(T(W),\delta)$ are  given by  proposition \ref{49}, satisfying:
 \begin{equation}
 H_{*}(\psi)=H_{*}(\alpha)\label{11}
 \end{equation}
\end{proposition}
\begin{proof}
Assume, by the way of the induction, that we have constructed a
dga-morphism $\alpha:(T(V_{\leq n}),\partial )\rightarrow
(T(W_{\leq n}),\delta)$ such that
 $\widetilde{\alpha}_{ k}'=\xi'_{k} ,\,\,\,\,\forall  k\leq
n$,  where $\widetilde{\alpha}_{ k}'$ is  the restriction  of
$\widetilde{\alpha}_{ k}$ to $\ker d_{k}$. Now  consider  the
following diagram which not need to be commutative:

\begin{picture}(300,130)(10,-5)
\put(60,100){$V_{n+1}\hspace{1mm}\vector(1,0){160}\hspace{1mm}W_{n+1}$}
 \put(49,76){$\beta_{n+1}$} \put(268,76){$\beta'_{n+1}$}
\put(76,96){$\vector(0,-1){38}$} \put(265,96){$\vector(0,-1){38}$}
\put(155,103){$\xi_{n+1}$} \put(165,52){$H_{n}(\alpha)$}
\put(60,48){$H_{n}(T(V_{\leq n}))
\hspace{1mm}\vector(1,0){100}\hspace{1mm}H_{n}(T(W_{\leq
n}))\vector(1,0){60}\hspace{1mm}W_{n}$}\put(325,52){$j'_{n}$}
 \put(60,30){$j_{n}$}\put(76,46){$\vector(0,-1){27}$}\put(70,8){$V_{n}$}
\end{picture}

\noindent Since $j'_{n}(H_{n}(\alpha) \beta_{n+1}-\beta'_{n+1}
\xi_{n+1})=0$, so $\mathrm{Im}\,(H_{n}(\alpha)
\beta_{n+1}-\beta'_{n+1}
\xi_{n+1})\subseteq \Gamma_{n}^{T(W_{\leq n})}$.

\noindent Recall that $d_{n+1}=j_{n}\circ \beta_{n+1}.$ \\

\noindent In order to continue the proof  we need  the following
technical  lemma:
\begin{lemma}
\label{lem1}There  exists a homomorphism $\omega_{n}:V_{n}\to
\Gamma_{n}^{T(W_{\leq n})}$ such that:
$$\mathrm{Im}\,\big(H_{n}(\alpha)\circ \beta
_{n+1}-\beta' _{n+1}\circ  \xi_{n+1}-\omega_{n}\circ
d_{n+1}\big)\subseteq \mathrm{Im}\,b'_{n+1}$$
\end{lemma}
\begin{proof}
Consider the following  diagram:

\begin{picture}(300,190)(0,-15)
\put(80,48){$(\mathrm{Im}\,d'_{n+1})'\overset{\varphi'_{n}}{\longrightarrow
}\Gamma^{T(W)}_{n}\overset{pr'}{\twoheadrightarrow}
\mathrm{Coker}\,b'_{n+1}$}
\put(80,102){$(\mathrm{Im}\,d_{n+1})'\overset{\varphi_{n}}{\longrightarrow
}\Gamma^{T(V)}_{n}\overset{pr}{\twoheadrightarrow}
\mathrm{Coker}\,b_{n+1}$}
 \put(95,98){$\vector(0,-1){39}$}
 \put(97,78){$\xi_{n+1}$}
 \put(160,98){$\vector(0,-1){39}$}
 \put(162,78){$\gamma^{\alpha}_{n}$}
\put(220,78){$\widetilde{\gamma}_{n}^{\alpha}$}
\put(218,98){$\vector(0,-1){39}$} \put(85,150){$\ker d_{n}$}
\put(85,0){$\ker d'_{n}$} \put(95,110){$\vector(0,1){37}$}
\put(95,46){$\vector(0,-1){37}$}
\put(118,150){$\vector(3,-1){100}$}
\put(115,8){$\vector(3,1){100}$} \put(162,138){$g_{n}$}
\put(162,16){$g'_{n}$} \put(73,127){\scriptsize $d_{n+1}$}
\put(73,30){\scriptsize $d'_{n+1}$}\put(7,78){$(E)$}
\end{picture}

\noindent where the top and the  bottom triangles commute by
virtues of relation (\ref {41}) in remark \ref {40}, where
$\varphi_{n}$ and $\varphi'_{n}$ are respectively the restriction
of the homomorphisms $\beta_{n+1}$ and $\beta'_{n+1}$  to
$(\mathrm{Im}\, d_{n+1})'$ and $(\mathrm{Im}\, d'_{n+1})'$. Recall
that the homomorphism
 $\widetilde{\gamma}_{n}^{\alpha}$ is defined by the formula
$\widetilde{\gamma}_{n}^{\alpha}(x+\mathrm{Im}\,
b_{n+1})=\gamma_{n}^{\alpha}(x)+\mathrm{Im}\, b'_{n+1}$.

\noindent First let us write $V_{n+1}\cong \ker d_{n+1}\oplus
(\mathrm{Im}\,d_{n+1})'$. As it has done in the proof of
proposition 3 and by using the formula (\ref{a30}) an easy
computation shows that on $\ker d_{n+1}$ the homomorphism $
H_{n}(\alpha)\circ \beta _{n+1}-\beta' _{n+1}\circ \xi_{n+1}$ is
nil, while on  on  $(\mathrm{Im}\,d_{n+1})'$ we get:
\begin{eqnarray}
H_{n}(\alpha)\circ \beta _{n+1}-\beta' _{n+1}\circ
\xi_{n+1}=\gamma _{n}^{\alpha}\circ \varphi_{n}-\varphi'_{n}\circ
\xi _{n+1}.\label{17}
\end{eqnarray}
But the diagram $(E)$ implies that:
\begin{eqnarray*}
 pr'\circ
(\gamma _{n}^{\alpha}\circ \varphi_{n}-\varphi'_{n}\circ \xi
_{n+1})&=&\widetilde{\gamma}_{n}^{\alpha}\circ pr\circ
\varphi_{n}-pr'\circ \varphi'_{n}\circ \xi _{n+1}\\
&=&\widetilde{\gamma}_{n}^{\alpha}\circ g_{n}\circ
d_{n+1}-g'_{n}\circ d'_{n+1}\circ \xi _{n+1}\\
&=&\widetilde{\gamma}_{n}^{\alpha}\circ g_{n}\circ
d_{n+1}-g'_{n}\circ \xi_{n}\circ d _{n+1}\\
&=&(\widetilde{\gamma}_{n}^{\alpha}\circ g_{n}\circ -g'_{n}\circ
\xi_{n})\circ d _{n+1}
\end{eqnarray*}
Therefore if we set
$\widetilde{\omega}_{n}=(\widetilde{\gamma}_{n}^{\alpha}\circ
g_{n}\circ -g'_{n}\circ \xi_{n}):\ker d_{n}\to
\mathrm{Coker}\,b'_{n+1}$, then the formula (\ref{17}) becomes:
\begin{eqnarray}
 pr'\circ
(H_{n}(\alpha)\circ \beta _{n+1}-\beta' _{n+1}\circ
\xi_{n+1})=\widetilde{\omega}_{n}\circ d _{n+1}\label{24}
\end{eqnarray}

\noindent Now since  $\ker d_{n}$ is a direct factor of $V_{n}$
which is free, there  exist two  homomorphism $\omega_{n}$ and
$\eta_{n}$ making the following diagram commutes:

\begin{picture}(300,80)(-30,-29)
\put(108,26){$\vector(3,-1){100}$}
\put(99,27){$\vector(0,-1){30}$}
\put(210,-12){$\vector(-1,0){60}$}
\put(215,-14){$\Gamma^{T(W_{\leq n})}_{n}$}
\put(70,-14){$0\leftarrow \mathrm{Coker}\,b'_{n+1}$}
\put(82,10){$\widetilde{\omega}_{n}$}
\put(40,30){$V_{n+1}\overset{d_{n+1}}{\longrightarrow}V_{n}\,\,\vector(1,0){80}\,\,Z_{n}(T(W_{\leq
n-1}))$} \put(161,10){$\omega_{n}$}
\put(170,-23){$pr'$}\put(150,33){$\eta_{n}$}\put(216,27){$\vector(0,-1){30}$}
\put(216,03){$\vector(0,-1){04}$}
\end{picture}

\noindent where  $%
Z_{n}(T(W_{\leq n-1}))$ is the sub-module of the $n$-cycles of
$T_{n}(W_{\leq n-1})$ and where $%
Z_{n}(T(W_{\leq n-1}))\to  \Gamma^{T(W_{\leq n})}_{n}$ is the
projection given by the relation (\ref{31}). Note that we have:
\begin{eqnarray}
 \omega_{n}(v_{n})=
\eta_{n}(v_{n}+\mathrm{Im}\,\partial)\label{a40}.
\end{eqnarray}
 So the formula (\ref{24}) becomes:
\begin{eqnarray}
 pr'\circ
(H_{n}(\alpha)\circ \beta _{n+1}-\beta' _{n+1}\circ
\xi_{n+1})=pr'\circ \omega_{n}\circ d _{n+1}\label{27}
\end{eqnarray}
The last formula means that  $pr'\circ (H_{n}(\alpha)\circ \beta
_{n+1}-\beta' _{n+1}\circ \xi_{n+1}- \omega_{n}\circ d _{n+1})$ is
nil in $\mathrm{Coker}\,b'_{n+1}$ so
$\mathrm{Im}\,\big(H_{n}(\alpha)\circ \beta _{n+1}-\beta'
_{n+1}\circ  \xi_{n+1}-\omega_{n}\circ d_{n+1}\big)\subseteq
\mathrm{Im}\,b'_{n+1}$ and the lemma is proved.
\end{proof}

 Now we are able to continue the proof of proposition 6. From lemma \ref{lem1}
we deduce that there exists a homomorphism $\lambda _{n+1}$ making
the following diagram commutes:

\begin{picture}(300,100)(-40,05)
\put(258,75){$\vector(0,-1){37}$} \put(246,78){$\ker d'_{n+1}$}
\put(77,34){$\vector(4,1){165}$}
\put(45,30){$V_{n+1}\hspace{1mm}\vector(1,0){169}\hspace{2mm}\mathrm{Im}\,b'_{n+1}$}
 \put(92,20){\scriptsize $H_{n}(\alpha) \beta
_{n+1}-\beta' _{n+1}  \xi_{n+1}-\omega_{n} d_{n+1}$}
 \put(143,60){$\lambda _{n+1}$}
\put(260,55){$\beta' _{n+1}$} \put(258,44){$\vector(0,-1){4}$}
\put(-20,55){$(F)$}
\end{picture}
\begin{remark}
\label{30} We have seen that the homomorphism $H_{n}(\alpha) \beta
_{n+1}-\beta' _{n+1}  \xi_{n+1}$ is nil on $\ker d_{n+1}$ so
$\lambda _{n+1}$ can be also chosen nil on $\ker d_{n+1}$.
\end{remark}

\noindent Define $\psi:T(V_{\leq
n}),\partial)\rightarrow(T(W_{\leq n}),\delta)$ by setting:
\begin{eqnarray*}
\psi_{n}=\alpha_{n}-h_{n}\hspace{2cm} \psi_{\leq n-1}=\alpha_{\leq
n-1}
\end{eqnarray*}
\begin{remark}
\label{35} Since $\mathrm{Im}\,h_{n}\subset Z_{n}(T(W_{\leq
n-1}))$, we deduce the following  properties:

\noindent first it is easy to see that $\psi$ is a dga-morphism,
next the  homomorphism $\widetilde{\psi}_{n}:V_{n}\rightarrow
W_{n}$ induced by $\psi$ satisfies the relation:
\begin{eqnarray}
\widetilde{\psi}_{k}'=\widetilde{\alpha}_{k}'=\xi'_{k}\,\,\,,\hspace{1cm}\forall
k\leq n\label{33}
\end{eqnarray}
where $\widetilde{\psi}_{ k}'$ is the restriction of
$\widetilde{\psi}_{ k}$ to  $\ker d_{k}$ and  finally we have:
\begin{eqnarray}
H_{n}(\psi)= H_{n}(\alpha)-\omega_{n}\circ j_{n}\label{34}.
\end{eqnarray}
The last formula is due to the following: let
$(v_{n}+q_{n}+\mathrm{Im}\,\partial)$ be a class in
$H_{n}(T(V_{\leq n}))$,  where $v_{n}\in V_{n}$ and where
$q_{n}\in T(V_{\leq n-1})$. Recall that the homomorphism
$j_{n}:H_{n}(T(V_{\leq n}))\to V_{n}$ is defined by
$j_{n}(v_{n}+q_{n}+\mathrm{Im}\,\partial)=v_{n}$.

\noindent According to the relation (\ref{40}) and the definition
of  $j_{n}$ and an easy computation shows that:
\begin{eqnarray*}
 (H_{n}(\alpha)-\omega_{n}\circ j_{n})(v_{n}+q_{n}+\mathrm{Im}\,\partial)&=&H_{n}(\alpha)(v_{n}+q_{n}+\mathrm{Im}\,\partial)-\\
 &&\omega_{n}\circ j_{n}(v_{n}+q_{n}+\mathrm{Im}\,\partial)\\
&=&(\alpha_{n}(v_{n}+q_{n})+\mathrm{Im}\,\delta)-\omega_{n}(v_{n})\\
&=&(\alpha_{n}(v_{n}+q_{n})+\mathrm{Im}\,\delta)-h_{n}(v_{n}+\mathrm{Im}\,\partial)\\
&=&H_{n}(\psi)(v_{n}+q_{n}+\mathrm{Im}\,\partial).
\end{eqnarray*}
\end{remark}

\noindent  Now we assert that
 the following diagram commutes:

\begin{picture}(300,90)(-30,30)
\put(60,100){$V_{n+1}\hspace{1mm}\vector(1,0){160}\hspace{1mm}W_{n+1}$}
 \put(49,76){$\beta_{n+1}$} \put(268,76){$\beta'_{n+1}$}
\put(76,96){$\vector(0,-1){38}$} \put(265,96){$\vector(0,-1){38}$}
\put(155,103){$\xi_{n+1}+\lambda_{n+1}$}
\put(165,52){$H_{n}(\psi)$} \put(60,48){$H_{n}(T(V_{\leq n}))
\hspace{1mm}\vector(1,0){100}\hspace{1mm}H_{3n}(T(W_{\leq
n}))\hspace{1mm}$}
\end{picture}

\noindent Indeed, by the diagram (F) and the relations (2.5) and
(\ref{34}) we can write:
\begin{eqnarray*}
H_{n}(\psi)\beta_{n+1}-\beta'_{n+1}( \xi
_{n+1}+\lambda_{n+1})\hspace{-3mm}&=&\hspace{-3mm}(H_{n}(\alpha)-\omega_{n}j_{n})
\beta_{n+1}-\beta'_{n+1} \xi _{n+1}-\beta'_{n+1} \lambda
_{n+1}\\
\hspace{-3mm}&=&\hspace{-3mm}H_{n}(\alpha)\beta_{n+1}-\omega_{n}j_{n}\beta_{n+1}-\beta'_{n+1} \xi_{n+1}\\
&&-\beta'_{n+1} \lambda_{n+1}\\
&=&\hspace{-3mm}H_{n}(\alpha)\beta_{n+1}-\omega_{n}d_{n+1}-\beta'_{n+1} \xi_{n+1}-\beta'_{n+1} \lambda_{n+1}\\
&=&0
\end{eqnarray*}
Therefore by lemma \ref{lem3} we can extend  $\psi$ to a dga-
morphism $\theta:(T(V_{n+1}),\partial)\rightarrow(T(W_{\leq
n+1}),\delta)$.

\noindent Finally by remark \ref{30} we know that $\lambda_{n+1}$
is nil on $\ker d_{n+1}$, it follows that
$\widetilde{\theta}'_{n+1}=\xi'_{n+1}$, where
$\widetilde{\theta}_{ n+1}'$ is the restriction of
$\widetilde{\theta}_{ n+1}$ to  $\ker d_{n+1}$, hence  by   the
relation $(\ref{33})$ we deduce that the formula $(\ref{X13})$ is
satisfied
\end{proof}

\begin{remark}
\label{21} For every $
 n$ we have
$H_{n}(\widetilde{\theta})= H_{n}(\xi_{*})$. Indeed; it is
well-known that:
\begin{eqnarray}
H_{n}(\widetilde{\theta})(x_{n}+\mathrm{Im}\,d_{n+1})=\widetilde{\theta}_{n}(n_{k})+\mathrm{Im}\,d'_{n+1}\label{36}
\end{eqnarray}
where $x_{n}\in \ker d_{n}$. But by the relation $(\ref{33})$ we
know that on  $\ker d_{n}$ we have
$\widetilde{\theta}_{n}=\xi_{n}$ so the formula $(\ref{36})$
becomes:
\begin{eqnarray}
H_{n}(\widetilde{\theta})(x_{n}+\mathrm{Im}\,d_{n+1})=\xi_{n}(x_{k})+\mathrm{Im}\,d'_{n+1}=H_{k}(\xi_{*})(x_{n}+\mathrm{Im}\,d_{n+1})
\end{eqnarray}
Therefore  $H_{n}(\widetilde{\theta})= H_{n}(\xi_{*})$.
\end{remark}

Now we are able to announce the following result which gives an
answer to the question asked in the beginning of this section:
\begin{theorem}
\label{7} Let $(T(V),\partial)$ and  $(T(W),\delta)$ two free
dgas.
 Assume that there exists a quasi-isomorphism
$\alpha:(T(V),\widetilde{\partial})\to (T(W),\widetilde{\delta})$
 between their associated perfect  dgas such that $\alpha:\big((T(V),\widetilde{\partial}),(\pi_{n})_{n\geq 2}\big)\to
\big((T(W),\widetilde{\delta}),(\pi'_{n})_{n\geq 2}\big)$ is a
morphism in the category $\mathbf {C}$. Then   $(T(V),\partial)$
and $(T(W),\delta)$ are quasi-isomorphic.
\end{theorem}
\begin{proof}
It is  follows immediately from  proposition 6 and Whitehead's
theorem.
\end{proof}
Recall that the notion of quasi-perfect  dgas is defined in
definition \ref{dd3}. As a consequence of theorem \ref{7} we
derive the following useful result:
\begin{corollary}
\label{t2} If  $(T(V),\partial)$ is  quasi-perfect, then
$(T(V),\partial)$ and  $(T(V),\widetilde{\partial})$ are
quasi-isomorphic.
\end{corollary}
\begin{proof}
On one hand since  $(T(V),\partial)$ is  quasi-perfect, we deduce
that the characteristic pair
$\big((T(V),\widetilde{\partial}),(\pi_{n})_{n\geq 2}\big)$
associated with $(T(V),\partial)$ is such that $\pi_{n}=0$ for all
$n\geq 2$. On the other hand  the pair
$\big((T(V),\widetilde{\partial}),(0)\big)$ can be considered as
 a characteristic pair associated with the perfect dga $(T(V),\widetilde{\partial})$. Therefore the morphism
$Id_{T(V)}:\big((T(V),\widetilde{\partial}),(\pi_{n})_{n\geq
2}\big)\to \big((T(V),\widetilde{\partial}),(0)\big)$, where
$Id_{T(V)}$ is the identity map which is obviously a
quasi-isomorphism, is a morphism in the category $\mathbf{C}$,
therefore
 by theorem \ref{7} we conclude that
 $(T(V),\partial)$ and  $(T(V),\widetilde{\partial})$ are
quasi-isomorphic.
\end{proof}
We have seen that  free dgas  $(T(V),\partial)$ such that
$H_{*}(V,d)$ is free are a kind of  quasi-perfect dgas and by
corollary \ref{t2} we know that quasi-perfect  dgas have the
quasi-isomorphism type of perfect dgas, therefore from theorem 1
we derive:
\begin{corollary}
\label{t2} If the two free dgas  $(T(V),\partial)$ and
$(T(W),\delta)$ are such that $H_{*}(V,d)$ and $H_{*}(W,d)$ are
free graded modules, then $(T(V),\partial)$ and $(T(W),\delta)$
are  quasi-isomorphic if and only if their Whitehead exact
sequences are isomorphic.
\end{corollary}

Corollary \ref{t2} can be interpreted geometrically as follows.
Recall that the Adams-Hilton model \cite{ada} of a simply
connected CW-complex is a quasi-isomorphism of algebras
$(T(s^{-1}\overline{CellX}),\partial)\overset{\simeq}{\longrightarrow}
C_{*}(\Omega X)$, where $C_{*}(\Omega X)$ is the singular chain
complex  of the loop space  of $X$ and where
$T(s^{-1}\overline{CellX})$ is the free dga of the free $
R$-module generated by the desuspension of the set of the cells of
$X$. As a consequence of corollary \ref{t2} we deduce the
following useful result:
\begin{corollary}
\label{34} If  $X$ and  $Y$ are two simply connected  CW-complex
such that  $H_{*}(X,R)$ and $H_{*}(Y,R)$
 are free,
then their respective  Adams-Hilton models are
 quasi-isomorphic if and only if the  Whitehead exact sequences
associated to these models  are isomorphic.
 \end{corollary}
\begin{proof}
According to the  properties of Adams-Hilton model we know that:
\begin{equation*}
 H_{*}(X,R)=
H_{*}(s^{-1}\overline{CellX},d)\,\,\, , \,\,\, H_{*}(Y,R)=
H_{*}(s^{-1}\overline{CellY},d')
\end{equation*}
 therefore the proof follows from  corollary \ref{t2}.
\end{proof}

Finally in  \cite{BH} Baues and Hennes  showed that there exist
4732 homotopy types of a simply connected  CW-complex  having the
following (reduced) homology groups ($n\geq4$):
$$
\begin{array}{lcrlcr}
\widetilde{H}_{n}(X)&=&\Bbb Z_{4}\oplus \Bbb Z_{4}\oplus \Bbb
Z\hspace{1cm}\widetilde{H}_{n+1}(X)&=&\Bbb Z_{8}\oplus \Bbb Z\\
\widetilde{H}_{n+2}(X)&=&\Bbb Z_{2}\oplus \Bbb Z_{4}\oplus \Bbb
Z\hspace{1cm}\widetilde{H}_{n+3}(X)&=&\hspace{-1cm}\Bbb Z
\end{array}
$$
By the  method developed  in this work  we  may  show, with  a
simple computation as we have done in   the example \ref{25}, that
there exist  36 quasi-isomorphic types of  the Adams-Hilton model
$(T(s^{-1}\overline{CellX}),\partial)$ of this space $X$. Note
that by hypothesis it is easy to check that the dga
 $(T(s^{-1}\overline{CellX}),\partial)$ is quasi-perfect.

If we denote by $\mathbf{C}_{\diagup_{\sim}}$ the set of all
isomorphism classes of objects of the category $\mathbf{C}$, by
$\mathbf{\hbar DGA}$ the set of all the quasi-isomorphism types of
objects of $\mathbf{DGA}$ and by $\mathbf{\hbar PDGA}$ the set of
all the quasi-isomorphism types of objects of $\mathbf{PDGA}$ then
propositions \ref{1}, \ref{2} and theorem \ref{7} allow us to
define  two surjective  applications:
$$\theta:\mathbf{C}_{\diagup_{\sim}}\to \mathbf{\hbar
 DGA}\hspace{2cm}\theta':\mathbf{\hbar DGA}\to \mathbf{\hbar PDGA}$$
by setting:
\begin{eqnarray*}
\theta( \{(T(V),\widetilde{\partial}),(\pi_{n})_{n\geq
2})\})=\{(T(V),\partial)\}\hspace{2cm}\theta'(
\{(T(V),\partial)\})=\{(T(V),\widetilde{\partial})\}
\end{eqnarray*}
so we can easily derive that: $$\text{ Cardinal }\mathbf{\hbar
PDGA} \leq \text{ Cardinal }\mathbf{\hbar DGA}\leq \text{ Cardinal
}\mathbf{C}_{\diagup_{\sim}}$$


\begin{thebibliography}{99}
\bibitem{ada}  J. F. Adams et P. J. Hilton, \emph{On the chain algebra of a
loop space, }Comm. Math. Helv. 30. 1955. p 305-330.


\bibitem{Ba2}  H.J. Baues, \emph{Homotopy Type and Homology, }Oxford
Mathematical Monographs. Oxford University Press, Oxford, 1996,
496p.

\bibitem{Ba3}  H.J. Baues, \emph{Algebraic homotopy}, Cambridge studies in
advanced mathematics 15, 1989.



\bibitem{BL}  H.J. Baues and J.M Lemaire, \emph{Minimal models in homotopy theory}, Math. Ann., 225, 219-242, 1977.

\bibitem{BH}  H.J. Baues and M. Hennes, \emph{ The homotopy classification of $(n-1)$-connected $(n+3)$-dimensional polydra, $n\leq 4$,
}Topology. Vol.30, No 3: 373-408, 1991.


	\bibitem{Benk10}M. Benkhalifa, \emph{On the group of self-homotopy equivalences of an elliptic space},
Proceedings of the American Mathematical Society (Article in press). DOI: https://doi.org/10.1090/proc/14900 

\bibitem{Benk11} M. Benkhalifa,  {\em The Adams-Hilton  model and  the group of self-homotopy equivalences   of a simply connected CW-complex
}. Homology, Homotopy and Applications, vol. 21(2), pp.345-362, 2019.

\bibitem{Benk12}M. Benkhalifa,      \emph{On the group of self-homotopy equivalences of $n$-connected and $(3n+2)$-dimensional CW-Complex},
Topology and its Applications, Volume 233, 1-15,
2018.		

\bibitem{Benk13} M. Benkhalifa,  {\em Postnikov decomposition and the the group of self-equivalences of a rationalized space
}. Homology, Homotopy and Applications, vol.19(1),  pp.209-224, 2017




\bibitem{Benk14}  M. Benkhalifa,  {\em Cardinality of rational homotopy types of simply connected CW-complexes,} Internat. J. Math. Vol. 22(N. 2) ,
179-193,  2011.

\bibitem{Benk15} M. Benkhalifa,   {\em Realizability of the group of rational self-homotopy  equivalences,}  J. Homotopy Relat. Struct.   5(1)
, 361-372, 2010



\bibitem{Benk18} M. Benkhalifa,   {\em The certain exact sequence of Whitehead and the classification of homotopy types of CW-complexes,} Topology and its Applications, 157 (14), 2240-2250,
2010.	

\bibitem{Benk16} M. Benkhalifa,
{\em Rational self-homotopy equivalences and Whitehead exact sequence,}
J. Homotopy Relat. Struct.  4(1),135-144, 2009







\bibitem{ben2}  M. Benkhalifa, \emph{Sur le type d'homotopie d'un
CW-complexe, }Homology, Homotopy and Applications. Vol.5 $(1)$:
101-120, 2003.

\bibitem{ben22}  M. Benkhalifa, \emph{On the homotopy type of a chain algebra, }Homology, Homotopy and Applications. Vol.6 $(1)$:
Volume 6, 109-137, 2004.

\bibitem{ben3}  M. Benkhalifa, \emph{Whitehead exact sequence and differential graded free Lie algebra, }
 International Journal of Mathematics, vol.15
(12),987-1005, 2004


\bibitem{bab}  M. Benkhalifa and N. Abughazalah, \emph{On the homotopy type of $(n-1)$-connected $(3n+1)$-dimensional free chain Lie algebra , }
 Central European Journal of Mathematics, Volume 3, N. 1,
58-75, 2005

\bibitem{BS} M. Benkhalifa and S. B. Smith,  {\em The effect of cell attachment on the group of self-equivalences of an
	R-local space}. Journal of Homotopy and Related Structures. Vol. 10(3),  549-564, 2015

\bibitem{MAL}  S. Mac lane, \emph{Homology}%
, Springer 1967.

\bibitem{M1}  J.C . Moore , \emph{S\'{e}minaire H. Cartan}%
, 1954-1955. Expos\'{e} 3.

\bibitem{W1}  J.H.C. Whitehead, \emph{A certain exact sequence}, Ann. Math,
52:51-110, 1950.

\end{thebibliography}
\end{document}